\documentclass[12pt]{amsart}
\usepackage{amsmath,amssymb,amsthm}    
\usepackage{mathabx}    
\usepackage{mathrsfs}    	
\usepackage[colorlinks]{hyperref}  
\usepackage[abbrev,msc-links]{amsrefs}  
\usepackage{enumerate}
\usepackage{enumitem}
\usepackage[symbol]{footmisc}
\usepackage[symbol]{footmisc}
\usepackage{pgf,tikz}
\usepackage{float}
\usepackage{mathrsfs}
\usetikzlibrary{arrows}
\usepackage{rotating}
\usetikzlibrary{calc}
\usetikzlibrary{positioning}
\usetikzlibrary{patterns}
\usetikzlibrary{shapes}
\usetikzlibrary{arrows}
\usetikzlibrary{snakes}
\usepackage{caption}
\usepackage{subcaption}

\normalsize                              
\newtheorem{theorem}{Theorem}[section]
\newtheorem{definition}[theorem]{Definition}
\newtheorem{lemma}[theorem]{Lemma}
\newtheorem{claim}[theorem]{Claim}

\newtheorem{corollary}[theorem]{Corollary}

\newtheorem{question}[theorem]{Question}
\newtheorem{conjecture}[theorem]{Conjecture}

\def\dHH#1{\leavevmode\setbox0=\hbox{#1}\dimen0=\wd0\setbox0=\hbox{.}%
	\advance\dimen0 by -\wd0%
	\hbox{#1\raise-0.5ex\hbox to 0pt{\hss.\kern.5\dimen0}}}%

\newcommand{\nn}{{\mathbb{N}}}

\newcommand{\s}[1]{\sigma_{#1}}

\begin{document}
	\title[Colourful matchings]{Colourful matchings}

	\author[A. Arman]{Andrii Arman}
	\address{Department of Mathematics, Emory University, Atlanta, GA 30322, USA}
	\email{andrii.arman@emory.edu}
	\thanks{}
	
	\author[V. R\"{o}dl]{Vojt\v{e}ch R\"{o}dl}
	\address{Department of Mathematics, 
		Emory University, Atlanta, GA 30322, USA}
	\email{rodl@mathcs.emory.edu}
	\thanks{The second  author was supported by NSF grant DMS 1764385}

	\author[M. T. Sales]{Marcelo Tadeu Sales}
	\address{Department of Mathematics, 
		Emory University, Atlanta, GA 30322, USA}
	\email{marcelo.tadeu.sales@emory.edu}
	
	
	\begin{center}
	\end{center}
	\begin{abstract} 
		Suppose a committee consisting of three members has to match $n$ candidates to $n$ different positions. Each member of the committee proposes a matching, however the proposed matchings totally disagree, i.e., every candidate is matched to three different positions according to three committee members. All three committee members are very competitive and want to push through as many of their suggestions as possible. Can a committee always find a compromise --- a matching of candidates to positions such that for every committee member a third of all candidates are assigned according to that committee member suggestion?
		
		We will consider an asymptotic version of this question and several other
		variants of similar problem. As an application we will consider an
		embedding problem --- in particular which configurations large Steiner
		systems always need to contain.
	\end{abstract}

	\maketitle
	
	
	\section{Introduction}
		In this paper we study the following question.
	\begin{question}\label{question:main}
		Let $G$ be a graph on $2n$ vertices that is a union of $k$ edge-disjoint perfect matchings $M_1,\ldots, M_{k}$. For which sequences $a_1, \ldots, a_{k}$ with $\sum_{i=1}^{k}a_i\leq n$ we can always find a new matching $M$ in $G$ such that $|M\cap M_{i}|\geq a_{i}$ for all $i\in[k]$? 
	\end{question}	
		Question~\ref{question:main}, in particular, is related to the famous Ryser conjecture, which is equivalent to asking if in every proper edge-colouring of $K_{n,n}$ there is a rainbow matching, i.e. if  for any $n$ disjoint perfect matchings $M_1,...,M_n$ of $K_{n,n}$ there is a matching $M$, such that $|M\cap M_{i}|=1$ for all $i\in[n]$. 
	
		Keevash, Pokrovskiy, Sudakov and Yepremyan~\cite{KPSY}  proved that if $K_{n,n}$ is decomposed into $n$ perfect matchings $M_1,\ldots, M_{n}$, then there is a matching $M$ (a rainbow matching) such that $|M\cap M_{i}|\leq 1$ for all $i\in[n]$ and $M$ has size at least $n-O(\frac{\log n}{\log \log n})$. 
	
		Note that the result of Keevash et.al.~\cite{KPSY} implies that relaxed version of Question~\ref{question:main} with $a_{1}=\ldots=a_{n}=1$ has affirmative answer: if $K_{n,n}$ is decomposed into $n$ disjoint perfect matchings, then there are $k=n-O(\frac{\log n}{\log \log n})$ of them, say $M_1, \ldots, M_{k}$ and a matching $M$ such that $|M\cap M_{i}|\geq 1$ for all $i\in[k]$.  

		The following is a central definition for this paper that formalizes a relaxation of Question~\ref{question:main} that we are considering.
	\begin{definition}\label{def:enhappy}
		We say that a sequence of integers $a_1, \ldots, a_k$ is {\textit{$(\varepsilon,n)$-happy}} if there is an integer $\ell\in[k,(1+\varepsilon)k]$ such that for any $\ell$ disjoint perfect matchings of $K_{2n}$ there are $k$ of them $M_1,\ldots, M_k$ and a matching $M$ such that $|M\cap M_i|\geq a_i$ for all $i\in[k]$.
	\end{definition}
		We show that a relaxed version of Question~\ref{question:main} has an affirmative answer for a wide variety of sequences, namely we prove the following theorem.
	
	\begin{theorem}\label{thm:main}
		For any $\varepsilon\in (0, 1)$ there is $n_0$, such that for all integers $n\geq n_0$ the following holds. 
		Let $\{a_i\}_{i=1}^{k}$ be a sequence of positive integer such that $\sum a_i<  (1-\varepsilon)n$ and either 
		\begin{itemize}
			\item[(i)] $a_i\leq n^{1-\varepsilon}$ for all $i$, or
			\item[(ii)] $a_i\geq n^{\varepsilon}$ for all $i$. 
		\end{itemize}
		Then sequence $\{a_{i}\}_{i=1}^k$ is $(\varepsilon,n)$-happy.
	\end{theorem}
	
		While the proof of case (i) of Theorem~\ref{thm:main} is based on the probabilistic argument the proof of case (ii) of Theorem~\ref{thm:main} relies on Alon's necklace theorem\cite{Alon}.
	
		We will find convenient to divide the proof of Theorem~\ref{thm:main} by formulating two separate statements.
	
	\begin{theorem}\label{thm:main1}
		For any $\varepsilon\in(0,1)$ there is $n_1$ such that for any $n\geq n_1$ any sequence $a_1, \ldots, a_k$ with each term at most $n^{1-\varepsilon}$ and such that $\sum_{i=1}^{k}a_{i}<(1-\varepsilon)n$ is $(\varepsilon,n)$-happy.
	\end{theorem}
		We deduce Theorem~\ref{thm:main1} from a slightly more technical Theorem~\ref{thm:main1:1} in Section~\ref{section:main1}.
	\begin{theorem}\label{thm:main2}
		For any $\varepsilon\in(0,1)$ there is $n_2$ such that for any $n\geq n_2$ any sequence $a_1, \ldots, a_k$ with each term at least $n^{\varepsilon}$ and such that $\sum_{i=1}^{k}a_{i}<(1-\varepsilon)n$ is $(\varepsilon,n)$-happy.
	\end{theorem}

		The proof of Theorem~\ref{thm:main2} relies on Theorem~\ref{thm:main1} and Alon's necklace theorem~\cite{Alon} and is presented in Section~\ref{section:main2}.
	
		For sequence $\mathcal{A}=\{a_1,\ldots,a_k\}$ such that $a_i\geq\alpha n$ for some real $\alpha$, it can be deduced from Theorem~\ref{thm:main2} that $\mathcal{A}$ is $(0,n)$-happy (see Corollary~\ref{cor:main}). In particular for the problem stated in the abstract, Corollary~\ref{cor:main} of Theorem~\ref{thm:main2} implies that for any $\varepsilon>0$ and large enough $n$, if disjoint perfect matchings $M_1, M_2, M_3$ of $K_{n,n}$ are given, then there is a matching $M$ such that $|M\cap M_i|\geq (\frac{1}{3}-\varepsilon)n$.

		In Section~\ref{section:STS} we show how Theorem~\ref{thm:main} is related to a problem of finding large hypertree in a Steiner triple system. A hypertree is a connected, simple $3$‐uniform hypergraph in which every two vertices are joined by a unique path. A Steiner triple system is a $3$‐uniform hypergraph in which every pair of vertices is	contained in exactly one edge. Easy greedy-type argument shows that any hypertree with at most $n$ vertices can be embedded into any Steiner triple system with $2n+1$ vertices. Elliot and the second author~\cite{ER} conjectured that the same is true even for larger hypertrees.
	\begin{conjecture}\label{conj:ER}
		Given $\varepsilon>0$ there is $n_{0}$, such that for any $n\geq n_0$, any hypertree $T$ of order $n$ and any Steiner triple system of order at least $(1+\varepsilon)n$, $S$ contains $T$ as a subhypergraph.
	\end{conjecture}	
		In~\cite{ER} authors proved Conjecture~\ref{conj:ER} for subdivision hypertree. With the proof based on Corollary~\ref{cor:main}, in Section~\ref{section:STS} we will verify the Conjecture~\ref{conj:ER} for another infinite family of hypertrees, which we call ``turkeys''.
	
	\section{Preliminaries}
	
	For positive integer $k$ let $[k]=\{1,\ldots, k\}$ and $[0,k]=\{0,\ldots, k\}$. We write $x=y\pm z$ if $x\in[y-z,y+z]$.
	
	We use the following version of Chernoff's bound (this is a corollary of Theorem 2.1 of Janson, \L uczak, Rucinski~\cite{JLR}).
	
	\begin{theorem}\label{thm:Chernoff}
		Let $X\sim \text{Bi}(n,p)$ be a binomial random variable with the expectation $\mu$, then for $t\leq 3\mu$
		$$\mathbb{P}(|X-\mu|>t)\leq 2e^{-t^{2}/(4\mu)}.$$
	\end{theorem}
	The proof of Theorem~\ref{thm:main1} relies on the application of a result by Alon, Kim and Spencer~\cite{AKS} on existence of almost perfect matching in 3-uniform simple hypergraphs. We use a version of this result stated by Kostochka and R\"{o}dl~\cite{KR}\footnote{We refer to Theorem 3 from that paper. There is a typo in the conclusion part of that theorem, where instead of $O(ND^{1/2}\ln^{3/2}D)$ there should be $O(ND^{-1/2}\ln^{3/2}D)$}.
	\begin{theorem}\label{thm:AKS}
		For any $K>0$ there exists $D_0$ such that the following holds. Let $H$ be a 3-uniform simple hypergraph on $N$ vertices and $D\geq D_0$ be such that $deg(x)=D\pm K\sqrt{D \ln D}$ for all $x\in V(H)$. Then $H$ contains a matching on $N-O(ND^{-1/2}\ln^{3/2}D)$ vertices. 
	\end{theorem}
	
	Here the constant in $O$-notation is depending on $K$ only an is independent of $N$ and $D$. 
	
	The following lemma is the main tool used for the proof of Theorem~\ref{thm:main1}. In Lemma~\ref{lemma:AKS2} the constant in $O$-notation is depending on $K$ and $\delta$ only an is independent of $m$ and $D$.	
	\begin{lemma}\label{lemma:AKS2}
		For any constant $K$ and $\delta$ there is $D_0$ such that the following holds for any $D\geq D_0$, $\Delta=K\sqrt{D\ln D}$ and an integer $m$. Let $G$ be a graph that is a union of edge disjoint matchings $M_1, \ldots, M_b$ such that:
		\begin{itemize}
			\item[(a)] $deg(v)=D\pm \Delta$ for all $v\in V(G)$.
			\item[(b)]  $|M_i|=Dm\pm \Delta m/2$ for all $i\in[b]$.
		\end{itemize}
		Then among matchings $M_1,\ldots, M_b$ there are at least $b^{'}=b\left(1-O( \frac{\ln^{3/2}D}{D^{1/2}})\right)$ of them, say $M_1, \ldots, M_{b^{'}}$ and a matching $M$ in $G$ such that $|M\cap M_i|\geq \frac{m}{1+\delta}$ for all $i\in [b']$.
	\end{lemma}
	\begin{proof}
	We start with proving the following quick consequence of Theorem~\ref{thm:AKS}.
	\begin{claim}\label{claim:AKS}
		For any constant $K$ there is $D_0$ such that the following holds for any $D\geq D_0$, $\Delta=K\sqrt{D\ln D}$ and integer $m$. Let $G$ be a graph that is a union of edge disjoint matchings $M_1, \ldots, M_b$ such that:
		\begin{itemize}
			\item[(a)] $deg(v)=D\pm \Delta$ for all $v\in V(G)$.
			\item[(b)]  $|M_i|=Dm\pm \Delta m/2$ for all $i\in[b]$
		\end{itemize}
		Then there are is a matching $M$ in $G$ containing at least $mb\left(1-O(\frac{\ln ^{3/2}D}{D^{1/2}})\right)$ edges such that $|M \cap M_i|\leq m$ for all $i\in[b]$.
	\end{claim}
	\begin{proof}[Proof of Claim~\ref{claim:AKS}]
		Let $D_0$ be provided by Theorem~\ref{thm:AKS} with a constant $K$ as input. Assume $D_0$ is large enough so that $D>4\Delta$. Let $G$ be a graph with vertex set that is a union of edge disjoint matchings $M_1,\ldots, M_b$. Number of edges in $G$ is equal to $|X|(D\pm \Delta)/2$ (by assumption (a))  and to $\sum_{i=1}^{b}|M_i|=mb(D\pm\Delta/2)$ (by (b)). Hence we may conclude that $|X|=2mb\pm mb$.
		
		For a given $G$ construct a 3-uniform hypergraph $H$ by arbitrarily splitting edges of each $M_i$ into $m$ groups $C_{i}^1, \ldots, C_{i}^{m}$ with $(D\pm \Delta)$ edges in each group (recall (b)). Note, that splitting $Dm\pm \Delta m/2$ edges into $m$ groups of size $(D\pm \Delta/2)$ might not be possible, as it may be the case that $|M_i|=Dm-\frac{\Delta m}{2}$ is an integer, but $D-\Delta/2$ is not.
		
		For each such group $C^j_{i}$ add a new vertex $y^j_{i}$ and construct $C^{j}_i$ hyperedges  (triples) by extending each edge in $C^{j}_i$ by the vertex $y_{i}^{j}$. Let $Y$ be a set of new vertices, then $|Y|=mb$ (as each $M_i$ produces $m$ new vertices). The hypergraph $H$ constructed in this way has $N=|X|+|Y|=\Theta(mb)$ vertices, is linear and each hyperedge has nonempty intersection with $Y$. Degree of each vertex of $H$ is $D\pm \Delta$, as every $y_{i}^j$ has degree $|C_{i}^j|=D\pm \Delta$ and degrees of vertices of $X$ in $H$ are same as in $G$.
		
		Apply Theorem~\ref{thm:AKS} to hypergraph $H$ (with $N=\Theta(mb)$) to obtain a 3-uniform matching $M'$ that misses at most $O(mbD^{-1/2}\ln^{3/2}D)$ vertices. In particular $M'$ misses at most $O(mbD^{-1/2}\ln^{3/2}D)$ vertices in $Y$. Recall that $|Y|=mb$ and each hyperedge of $H$ uses only one vertex of $Y$, so $M'$ has at least $mb-O(mbD^{-1/2}\ln^{3/2}D)$ edges. Let $M$ be a (2-uniform) matching obtained from $M'$ by deleting a vertex in $Y$ from each hyperedge of $M'$. Then $$|M|=|M'|=mb-O(mbD^{-1/2}\ln^{3/2}D).$$ For each $i\in [b]$, matching $M$ intersects $M_i$ in at most $m$ edges, as edges of $M_i$ are extended to a union of $m$ stars in $H$ centered at vertices $y_{i}^j$. Therefore, $M$ is a matching such that contains at least $mb-O(mbD^{-1/2}\ln^{3/2}D)$ edges and such that $|M\cap M_i|\leq m$ for all $i\in[b]$.
	\end{proof}
		Now, to prove Lemma~\ref{lemma:AKS2} we first apply Claim~\ref{claim:AKS} to the graph $G$ and obtain a matching $M$ such that $|M|=mb\left(1-O(\frac{\ln^{3/2}D}{D^{1/2}})\right)$and $|M\cap M_i|\leq m$.
		 
		Let $A_i=M\cap M_i$, and for each $A_i$ we define a ``deficit'' to be $d_i=m-|A_i|$. The total deficit of $A_1,\ldots, A_b$ is 
		$$\sum_{i=1}^{b} d_i=\sum_{i}^{b} \left(m-|A_i|\right)=mb-|M|=O\left(mb\frac{\ln ^{3/2}D}{D^{1/2}}\right).$$
		We now will find many sets $A_i$ with small deficit.
		
		We say that set $A_i$ is ``bad'' if it has size smaller than $m/(1+\delta)$, or in other words if its deficit $d_i$ is more than $\delta m/(1+\delta)$. Since the total deficit is $O(mb\frac{\ln ^{3/2}D}{D^{1/2}})$, the number of ``bad'' $A_i$'s is at most 
		$$\left. O\left(mb\frac{\ln ^{3/2}D}{D^{1/2}}\right)\middle/ \frac{\delta m}{1+\delta}\right.= O\left(b\frac{\ln ^{3/2}D}{D^{1/2}}\right).$$
		
		Finally, we deduce that there is a collection of at least $b^{'}=b-O(b\frac{\ln ^{3/2}D}{D^{1/2}})$ ``good'' sets $A_i$, say $A_1,\ldots, A_{b^{'}}$, such that $A_i\geq m/(1+\delta)$ for $i\in[b^{'}]$. This means that for matching $M$ and matchings $M_1, \ldots, M_{b^{\prime}}$ we have $|M\cap M_i|\geq \frac{m}{1+\delta}$, finishing the proof of the Lemma.
	\end{proof}

	\section{Proof of Theorem~\ref{thm:main1}}\label{section:main1}
	Theorem~\ref{thm:main1} follows from a slightly more technical Theorem~\ref{thm:main1:1}.
	\begin{theorem}\label{thm:main1:1}
		For any $\varepsilon<\frac{1}{10}$ there is $n_0$ such that for any $n\geq n_0$ the following holds. Let $\{a_{i}\}_{i=1}^{k}$ be a sequence such that
		\begin{itemize}
			\item[(a)] $\sum_{i=1}^{k}a_{i}<\left(1-\frac{\varepsilon}{2}\right)n$,
			\item[(b)] $a_i\leq n^{1-\varepsilon}$ for all $i\in[k]$,
			\item[(c)] $k\geq \frac{1}{3}n^{\varepsilon}$,
		\end{itemize}
		and additionally, if $k<n^{1/2+\varepsilon}$
		\begin{itemize}
			\item[(d)] $a_i\geq k^{-1}n^{1-\varepsilon/400}$ for all $i\in[k]$.
		\end{itemize}
		Then sequence $\{a_{i}\}_{i=1}^{k}$ is $(\varepsilon,n)$-happy.
	\end{theorem}
	We now show how to deduce Theorem~\ref{thm:main1} from Theorem~\ref{thm:main1:1}.
	\begin{proof}[Proof of Theorem~\ref{thm:main1}]
		Assume $\varepsilon$ is given, we will specify $n_1$ implicitly. Let $a_1, \ldots, a_k$ be a sequence of integers such that each $a_i\leq n^{1-\varepsilon}$ and $\sum_{i=1}^{k}{a_i}<(1-\varepsilon)n$. We start by noticing that if a sequence is $(\varepsilon_0,n)$-happy then it is also $(\varepsilon,n)$-happy for all $\varepsilon\geq \varepsilon_0$. Consequently it is sufficient to prove Theorem~\ref{thm:main1} for 
		$\varepsilon<\frac{1}{10}.$
		
		We may also assume that $k\geq \frac{1}{3}n^{\varepsilon},$ otherwise $\sum_{i=1}^{k}a_i<\frac{1}{3}n$ and it is not hard to observe that we can choose desired sets $A_{i}$ greedily from any collection of $k$ matchings.
		
		All together, we may assume that $\varepsilon$ and a sequence $\{a_{i}\}_{i=1}^{k}$ satisfy the following conditions
		\begin{equation*}
		a_i\leq n^{1-\varepsilon} \text{ for all } i\in[k], \qquad \varepsilon<\frac{1}{10}, \qquad k\geq \frac{1}{3}n^{\varepsilon}, \qquad \sum_{i=1}^{k}{a_i}<(1-\varepsilon)n.
		\end{equation*}
		
		So, sequence $\{a_{i}\}_{i=1}^{k}$ satisfies assumptions (a),(b),(c) of Theorem~\ref{thm:main1:1}. 
		
		Therefore, if $k\geq n^{1/2+\varepsilon}$, by Theorem~\ref{thm:main1:1}, sequence $\{a_{i}\}_{i=1}^{k}$ is $(\varepsilon,n)$-happy.
		
		For the rest of the proof assume $k<n^{1/2+\varepsilon}$. The sequence $\{a_{i}\}_{i=1}^{k}$ might not satisfy condition (d) of Theorem~\ref{thm:main1:1}, so we will construct a new sequence $\{a_{i}^{'}\}_{i=1}^{k}$ that satisfies assumptions (a)-(d) of Theorem~\ref{thm:main1:1} and such that $a_{i}^{'}\geq a_{i}$ for all $i\in[k]$. Define  $$\delta=\varepsilon/100, \qquad \lambda=\varepsilon/400.$$ Let 
		\begin{equation}\label{ineq:a1}
		a=\left\lceil k^{-1}n^{1-\lambda}\right\rceil.
		\end{equation}
		Now $k\geq \frac{1}{3}n^{\varepsilon}$ by~(\ref{ineq:many}), and we have $k^{-1}n^{1-\lambda}<n^{1-\varepsilon}-1$ for sufficiently large $n$, so 
		\begin{equation}\label{ineq:a1.5}
			a\leq n^{1-\varepsilon}.
		\end{equation}
		
		Consider a new sequence $\{a^\prime_i\}_{i=1}^{k}$, that is obtained from $\{a_i\}_{i=1}^{k}$ by setting for all $i\in[k]$ $$a^{\prime}_{i}=\max\{a_i, a\}.$$
		
		Note that according to this definition, for all $i\in [k]$ we have $$a_i^{\prime}\geq a\stackrel{(\ref{ineq:a1})}{\geq} k^{-1}n^{1-\lambda},$$
		so sequence $\{a^{\prime}_{i}\}_{i=1}^{k}$ satisfies assumption~(d) of Theorem~\ref{thm:main1:1}. The sequence $\{a^{\prime}_{i}\}_{i=1}^{k}$ also satisfies condition~(c) of Theorem~\ref{thm:main1:1}, as we assumed $k\geq \frac{1}{3}n^{\varepsilon}$.
		\begin{claim}\label{claim:ai'}
			Sequence $\{a^{'}_{i}\}_{i=1}^{k}$ satisfies the following properties:
			\begin{itemize}
				\item[(a)] $\sum_{i=1}^{k}a'_{i}<(1-\frac{\varepsilon}{2})n.$
				\item[(b)] $a'_{i}\leq n^{1-\varepsilon} \text{ for all } i\in[t].$
				\item[(c)] If the sequence $\{a^{'}_{i}\}_{i=1}^{k}$ is $(\varepsilon,n)$-happy, then the sequence $\{a_{i}\}_{i=1}^{k}$ is also $(\varepsilon,n)$-happy.
			\end{itemize}
		\end{claim}
		\begin{proof}
			{\bf Part (a).} Since $\sum a_i< (1-\varepsilon)n$ and $k<n^{1/2+\epsilon}$, for large $n$ we have
			$$\sum_{i=1}^{k}a^{\prime}_i\leq \sum_{i=1}^{k}(a+a_i)=ka+\sum_{i=1}^{k}a_i\stackrel{\text{(\ref{ineq:a1})}}{<} n^{1-\lambda}+k+\sum_{i=1}^{k}a_i<\left(1-\frac{\varepsilon}{2}\right)n.$$
			
			{\bf Part (b).} For each $i\in[k]$ we have $a_i\leq n^{1-\varepsilon}$, so 
			$$a'_i=\max\{a_i,a\}\stackrel{(\ref{ineq:a1.5}) }{\leq} n^{1-\varepsilon}.$$
			
			{\bf Part (c).} Since $a_{i}^{^\prime}\geq a_{i}$ for all $i\in[k]$, the Definition~\ref{def:enhappy} ensures that $(\varepsilon,n)$-happiness of $\{a_{i}^{\prime}\}_{i=1}^{k}$ implies $(\varepsilon,n)$-happiness of $\{a_{i}\}_{i=1}^{k}$. 
		\end{proof}
		Finally, by Claim~\ref{claim:ai'}, sequence $\{a_{i}^{\prime}\}_{i=1}^{k}$ satisfies assumptions (a), (b) of Theorem~\ref{thm:main1:1}, so it is $(\varepsilon,n)$-happy.
		Therefore, by Claim~\ref{claim:ai'}(c), sequence $\{a_i\}_{i=1}^{k}$ is also $(\varepsilon,n)$-happy.
		
		This finishes the proof of Theorem~\ref{thm:main1}.
	\end{proof}
	
	For the rest of this section we concentrate on proving Theorem~\ref{thm:main1:1}
	
	{\subsection{Choice of constants} Assume $\varepsilon<\frac{1}{10}$ is given, we will specify $n_1$ implicitly.
	Define auxiliary constants
	$$\delta=\frac{\varepsilon}{100}, \qquad \lambda = \frac{\varepsilon}{400}=\frac{\delta}{4}, \qquad K=\frac{4}{\lambda}.$$	
	Let sequence $\{a_{i}\}_{i=1}^{k}$ satisfy assumptions of Theorem~\ref{thm:main1:1}, that is
	\begin{equation}\label{ineq:many}
	 \sum_{i=1}^{k}{a_i}<\left(1-\frac{\varepsilon}{2}\right)n, \qquad a_i\leq n^{1-\varepsilon} \text{ for all } i\in[k], \qquad  k\geq \frac{1}{3}n^{\varepsilon},
	\end{equation}
	and in case $k< n^{1/2+\varepsilon}$ additionally  
	\begin{equation}\label{ineq:scase}
		a_i\geq k^{-1}n^{1-\lambda} \text{ for all } i\in[k].
	\end{equation}
	For the rest of the proof whenever we use $O$-notation, the implicit constant is dependent on $\varepsilon$ only.

	\subsection{Proof plan of Theorem~\ref{thm:main1:1}.} In order to obtain desired sets $A_i=M\cap M_i$ for $i\in[k]$ we are going to apply Lemma~\ref{lemma:AKS2}. To illustrate this we first consider following simple case when sequence $\{a_i\}_{i=1}^k$ is almost regular, that is such that for all $i\in[k]$ 
	$$a_i\in \left(\frac{n^{1-\varepsilon}}{1+\delta}, n^{1-\varepsilon}\right]$$ and $\sum_{i=1}^{k} a_i<(1-\varepsilon/2)n$. Let $\ell=(1+\varepsilon)k$ and  assume that $X$ is a set of size $2n$, and $\ell$ disjoint perfect matchings $M_1, \ldots, M_\ell$ on $X$ are given. Let $G$ be a graph formed by matchings $M_1, \ldots, M_{\ell}$. Our goal is to find $k$ matchings, say $M_1, \ldots M_k$ and sub-matchings $A_i\subseteq M_i$ such that $|A_i|\geq a_i$ and $\bigcup_{i=1}^{k} A_i$ is a matching. 
	
	We will apply Lemma~\ref{lemma:AKS2} to a subgraph $G'$ of $G$ with a large constant $K$, given $\delta$ and $m=(1+\delta)n^{1-\varepsilon}$. To this end let $G'$ be a random induced subgraph of $G$, where each vertex of $G$ is chosen to be in $G'$ independently with probability $p=\frac{m\ell}{n}$ (note that $p<1$). 
	
	Since $deg_{G}(v)=\ell$ for any $v\in V(G)$ we expect the degree of $x$ in $G'$ to be $p\ell$. Let $M_i'=M_i\cap E(G')$, then we expect $|M_i'|$ to be $p^2n=p\ell m$. Therefore for $D=p\ell$, $\Delta=K\sqrt{D\ln D}$ and $b=\ell$, as an easy consequence of Chernoff bound, we infer that with a positive probability all of the following properties hold for $G'$:
	\begin{itemize}
		\item $deg_{G'}(v)=D\pm \Delta$ for all $v\in V(G')$.
		\item  $|M'_i|=Dm\pm \Delta m/2$ for all $i\in[b]$.
	\end{itemize}
	Now, apply Lemma~\ref{lemma:AKS2} to $G'$ to obtain a collection of $b'=\ell(1-O(\frac{\ln^{3/2}D}{D^{1/2}}))$ vertex disjoint matchings, say  $A_1, \ldots, A_{b'}$ such that $A_i\subseteq M_i$ and  $|A_i|\geq m/(1+\delta)$ for $i\in[b']$. It remains to notice that $\ell\geq k \geq \frac{1}{3}n^{\varepsilon}$, $p\geq \frac{1}{3}$ and so $D=p\ell\geq \frac{1}{9}n^{\varepsilon}$. So for large enough $n$ we have $b'\geq k$. Therefore matchings $A_1,\ldots, A_k$ are such that $|A_i|\geq n^{1-\varepsilon}\geq a_i$ for all $i$ and $\bigcup_{i=1}^{k}A_i$ is a matching.
	
	For a general sequence $\{a_i\}_{i=1}^{k}$ we want to adopt the approach of the previous ``simple'' case when all the $a_i$'s are almost equal. 
	To this end we will group $a_i$'s in $t=O(\ln n)$ sets $I_1,\ldots I_t$, so that inside each set $I_j$ the elements do not differ much (this is done in Step 1 of formal proof). This way we will have for each group similar starting position
	as in the ``simple'' case. Let $n_j$ be the number of elements of sequence $\{a_i\}_{i=1}^{k}$ in $I_j$. For notation convenience let the elements of $\{a_i\}_{i=1}^{k}$ in $I_j$ form a sequence $\{a^{i}_{j}\}_{i=1}^{n_j}$.
	
	Next, we define two sequences $m_1,\ldots, m_t$ (in Step 1) and $b_1,\ldots, b_t$ (in Step 2) that will be used as an input values for $m$ and $b$ in Lemma~\ref{lemma:AKS2}. Some technical but crucial properties of those sequences are: 
	\begin{itemize}
		\item[(a)]  $\lceil m_j/(1+\delta)\rceil\geq \max\{x: x\in I_j\}$ for all $j$.
		\item[(b)] 	$b_j\geq n_j+m_j^{-1/2}n^{\lambda+1/2}$ if $n_j\neq 0$.
	\end{itemize}  
 	Here we also choose $\ell=\sum b_j$ and in Claim~\ref{claim:l-k} we show that $\ell\leq (1+\varepsilon)k$.

Finally in Step 3 of the proof, for $G$ that is a union of $\ell$ matchings on set $X$ we construct subgraphs $G_1, \ldots, G_t$ of $G$ in the following way. We first partition collection of matchings $M_1,\ldots, M_\ell$ into sub-collections $\mathcal{M}_1, \ldots, \mathcal{M}_t$ so that each $|\mathcal{M}_j|=b_j$. We also set, after relabeling, $\mathcal{M}_{j}=\{M_j^{1}, \ldots, M_{j}^{b_j}\}$. Then we partition vertex set $X$ into sets $X_0,\ldots, X_t$ (vertices of $X_0$ are leftover vertices that will not be used). Finally  $G_j$ is a defined to be a graph induced by edges of $\mathcal{M}_j$ on $X_j$. 

In Claim~\ref{claim:split} it is shown that we can find partition $X$ so that every $G_j$ satisfies all conditions necessary to apply Lemma~\ref{lemma:AKS2}. Claim~\ref{claim:split} is proved by considering a random partition of $X$. We then apply Lemma~\ref{lemma:AKS2} to a graph $G_j$ with $m=m_j$, $b=b_j$ and $D=D_j=b_j^2m_j/n$. 
	 	
	An application of Lemma~\ref{lemma:AKS2} to each $G_j$ yields the following: there are at least $b_j'=b_j\left(1-O( \frac{\ln^{3/2}D_j}{D_j^{1/2}})\right)$ matchings, say $M^{1}_j, \ldots, M^{b_j'}_{j}$ and sets $A^i_{j}\subseteq M^i_{j}$ for $i\in [b_j']$ such that  $|A^i_{j}|\geq \frac{m_j}{1+\delta}$ fo all $i\in[b_j']$ and that $\bigcup_{i=1}^{b_j'} A^i_j$ is a matching in $G_j$. 

We will prove, that our choice of $b_j$ (see (b) above) will guarantee for large enough $n$, that $b'_j\geq n_j$ for each $j\in[t]$. As mentioned above there are $b'_j$ sets $A_{j}^{i}$'s. On the other hand, $n_j$ is equal to the number of elements of sequence $\{a_{i}\}_{i=1}^{k}$ in set $I_j$ (also recall that those elements were labeled as  $\{a^{1}_{j}, \ldots, a_{j}^{n_j}\}$). So every element $a^{i}_{j}\in I_j$ with $i\in [n_j]$ can be put in correspondence with the matching $A_{j}^{i}$. 

By Lemma~\ref{lemma:AKS2} we have that $|A_{j}^{i}|\geq m_j/(1+\delta)$. Since $a_{j}^{i}\in I_{j}$, it follows from property (a) that $|A_{j}^{i}|\geq a^{i}_{j}$ for all $j\in[t]$, $i\in [n_j]$. Finally, $M=\bigcup_{j=1}^{t} \bigcup_{i=1}^{n_j} A_{j}^{i}$ is a desired matching since all $A_{j}^i$'s are vertex disjoint.

	\subsection{Formal proof of Theorem~\ref{thm:main1:1}.} Recall that the proof of Theorem~\ref{thm:main1:1} is divided into three steps: construction of sets $I_1, \ldots, I_t$ and sequence $\{m_j\}_{j=1}^{t}$; construction of  sequence $\{b_{j}\}_{j=1}^{t}$; and construction of subgraphs $G_1, \ldots, G_t$ with a subsequent application of Lemma~\ref{lemma:AKS2} to each of $G_j$.
		
	{\bf Step 1:} In this step we will construct sets $I_1,\ldots, I_t$ and a sequence $\{m_1,\ldots, m_t\}$.
	 Let $I$ be the set of integers in $[1, n^{1-\varepsilon}]$, we cover $I$ by sets $I_1, \ldots, I_t$. Recall that $\delta=\varepsilon/100$, $\lambda=\delta/4$ and choose an integer $$C=\left\lceil\frac{1}{\delta}\right\rceil.$$ There will be a special treatment of integers that are smaller than $C$. 
	 \begin{definition}\label{def:Ij}
	 	Define the sets $$I_j=\{j\} \text{ for } j\in [C]$$
	 	and for $j\in \nn$ define 
	 	$$I_{C+j}=\left(C(1+\delta)^{j-1}, C(1+\delta)^{j}\right] \cap \mathbb{N}.$$
	 \end{definition}
 	 Define $t$ to be the smallest integer such that $C(1+\delta)^{t-C}\geq n^{1-\varepsilon}$. 
 	 Notice that the choice of $C$ guarantees that every $I_j$ for $j\in [t]$ is nonempty.  Choice of $t$ guarantees that $I_1, \ldots, I_t$ cover $I=[1,n^{1-\varepsilon}]$ and  that $t=O(\ln n)$. 
		
	\begin{definition}\label{def:mj}
		Define $m_j=j$ if $j\leq C$. If $j>C$ define
		\begin{equation*}
		m_{j}=\min\{x\in \mathbb{N}: x\in I_{j+2}\}.
		\end{equation*} 
	\end{definition} 
	Notice that $\{m_{j}\}_{j=1}^{t}$ is an increasing sequence. We also need the following observation about sequence $\{m_j\}_{j=1}^{t}$.
	\begin{claim}\label{claim:ineq:gammaj}
		 If $j\in[C+1,t]$ and $x\in I_{j}$, then 
		  $$m_j\in [(1+\delta)x, (1+\delta)^3x].$$
		 If $x\in I_j$ with $j\leq C$, then $m_j=x=j$.
	\end{claim}
	\begin{proof}
	 If $x\in I_{j}$ with $j>C$, then
	$$C(1+\delta)^{j-C-1}\leq x\leq C(1+\delta)^{j-C},$$
	$$C(1+\delta)^{j-C+1}\leq m_{j}\leq C(1+\delta)^{j-C+2}.$$
	
	 So if $x\in I_j$ for $j>C$ then
	 $$(1+\delta)x \leq m_j \leq (1+\delta)^3x.$$
	The second statement of the Claim is trivial since $I_j=\{j\}$ and $m_{j}=j$ when $j\leq C$.
	\end{proof}
	
	{\bf Step 2:} The second step of the proof consists of constructing sequences $n_1, \ldots, n_{t}$ and $b_1,\ldots, b_t$. 
	
		\begin{definition}\label{def:njaij}
		Define $n_j$ to be the number of elements of $\{a_i\}_{i=1}^{k}$ in set $I_j$ for all $j\in[t]$. In other words
			$$n_{j}=|\{i\in [k]: a_{i}\in I_{j}\}|.$$
		
		Define a sequence $\{a^{i}_{j}\}_{i=1}^{n_j}$ to consist of all elements of $\{a_i\}_{i=1}^{k}$ that are in $I_j$.
	\end{definition}
	Notice that according to this definition, a sequence $\{a_{i}\}_{i=1}^{k}$ is partitioned into $t$ sequences $\{a^{i}_{j}\}_{i=1}^{n_{j}}$ with $j\in [t]$. Additionally we have 
		\begin{equation}\label{eq:k}
		k=\sum_{j=1}^{t}n_j.
		\end{equation}
	
	Now we define sequence $b_{1}, \ldots, b_{t}$. Recall that $\lambda=\delta/4$.
	
	\begin{definition}\label{def:bj}
		For $j\in [t]$, set  $b_j=n_j+\lceil m_j^{-1/2}n^{\lambda+1/2}\rceil$  if $n_j\neq0$, and set $b_j=0$ if $n_j=0$.
	\end{definition}	
	Finally, we set 
	\begin{equation}\label{eq:ell}
		\ell=\sum_{j=1}^{t}b_j.
	\end{equation}
	The following Claim proves that $\ell$ is at most $(1+\varepsilon)k$.
	\begin{claim}\label{claim:l-k}
		We have $\ell-k \leq \varepsilon k$. 
	\end{claim}
	\begin{proof}
		First, $\ell-k$ is equal to $\sum_{j=1}^{t}(b_j-n_j)$ by~(\ref{eq:k}) and~(\ref{eq:ell}). For all $j\in [t]$, according to Definition~\ref{def:bj},
		$$b_j-n_j \leq   m_j^{-1/2}n^{\lambda+1/2}+1.$$ 
		
		{\bf Case 1.} Consider the case $k\geq n^{1/2 +\varepsilon}$ first. Since each $m_j$ is a positive integer,
		$$\ell-k=\sum_{j=1}^{t}{b_j-n_j} \leq \sum_{j=1}^{t}1+ n^{\lambda+1/2}m_j^{-1/2} \leq t+t n^{\lambda+1/2}.$$
		Recall that $t=O(\ln n)$ and $\lambda<\varepsilon$, so for large enough $n$,
		$$\ell-k=O(\ln n)n^{\lambda+1/2}\leq \varepsilon n^{1/2+\varepsilon}\leq \varepsilon k.$$
		This finishes the proof of the Claim for case $k\geq n^{1/2+\varepsilon}$.
			
		{\bf Case 2.} Now we consider the case $k< n^{1/2 + \varepsilon}$. In this case, $k^{-1}n^{1-\lambda}\geq n^{1/2-\varepsilon-\lambda}$ and sequence $\{a_{i}\}_{i=1}^{k}$ satisfies property~(\ref{ineq:scase}), so for all $i\in[k]$ we have
		\begin{equation}\label{ineq:knmu}
		a_i\geq n^{1/2-\varepsilon-\lambda}.
		\end{equation}
		Next we will define an integer $j_0$ and show that $n_j=0$ for all $j<j_0$. 
		
		To this end consider the smallest $j_0\in[t]$, such that 
		\begin{equation}\label{ineq:gammaj0:1}
		m_{j_0}\geq (1+\delta)k^{-1}n^{1-\lambda}.
		\end{equation}
		In order to observe that such $j_0$ exists we notice that $k\geq \frac{1}{3}n^{\varepsilon}$ (by (\ref{ineq:many})) and so by Definition~\ref{def:mj} we have for sufficiently large $n$ $$m_t\geq n^{1-\varepsilon}\geq (1+\delta)k^{-1}n^{1-\lambda}.$$ Therefore there is $j_0\leq t$ that satisfies (\ref{ineq:gammaj0:1}). 
		
		For $i\in [k]$, by~(\ref{ineq:knmu}) we have $a_{i}\geq n^{1/2-\varepsilon-\lambda}$, so for large enough $n$, 
		if $a_{i}\in I_j$, then $j>C$ (see Definition~\ref{def:Ij}).
		
		Therefore if $a_{i}\in I_{j}$, by applying Claim~\ref{claim:ineq:gammaj} with $x=a_{i}$, we get 
		$$m_{j}\geq (1+\delta)a_i\stackrel{\text{(\ref{ineq:scase})}}{\geq} (1+\delta)k^{-1}n^{1-\lambda}.$$
		
		Comparing this inequality with definition of $j_0$~(\ref{ineq:gammaj0:1}) we get that if $a_{i}\in I_j$ for some $i\in[k]$, then $j\geq j_0$. 
		
		Therefore for all $j<j_0$ there are no elements of $\{a_{i}\}_{i=1}^{k}$ in sets $I_j$, which implies that $n_j=0$ and therefore $b_{j}=0$ (by Definition~\ref{def:bj}) for all $j<j_0$. 
		
		Finally, since $\{m_{j}\}_{j=1}^{t}$ is increasing,  provided  $n$ is large enough
		$$\ell-k=\sum_{j=1}^{t}b_j-n_j=\sum_{j=j_0}^{t}b_j-n_j \stackrel{\text{Def }\ref{def:bj}}{\leq} \sum_{j=j_0}^{t}1+ n^{\lambda+1/2}m_j^{-1/2} \leq t+t n^{\lambda+1/2} m_{j_0}^{-1/2}.$$
		Using $t=O(\ln n)$ and inequality~(\ref{ineq:gammaj0:1}), for large enough $n$
		$$\ell-k=O(\ln n)n^{\lambda+1/2}(k^{1/2}n^{\lambda/2-1/2})=O(\ln n)k^{1/2}n^{3\lambda/2}\leq k^{1/2}n^{3\varepsilon/8}.$$ 
		Now, recall that by~(\ref{ineq:many}) we have $k\geq \frac{1}{3}n^{\varepsilon}$, so
		$$\ell-k\leq 3k^{1/2}k^{3/8}<\varepsilon k$$
		for large enough $n$. So $\ell-k\leq \varepsilon k$ in both cases.

		\end{proof}
	
		The next Claim will be used in Step 3 of the proof.
		\begin{claim}\label{claim:ineq:p_j}
			For large enough $n$, $$\sum_{j=1}^{t}b_jm_j\leq \left(1-\frac{\varepsilon}{4}\right)n.$$
		\end{claim}
		\begin{proof}
	    According to Definition~\ref{def:bj} 
		$$\sum_{i=1}^{t}b_jm_j \leq \sum_{i=1}^{t}n_jm_j+\sum_{j=1}^{t}\left(1+ n^{\lambda+1/2}m_j^{-1/2}\right)m_{j}.$$
		Let $S_1=\sum_{i=1}^{t}n_jm_j$ and $S_2=\sum_{j=1}^{t}\left(1+ n^{\lambda+1/2}m_j^{-1/2}\right)m_{j}$.
		
		We now estimate $S_1$. Claim~\ref{claim:ineq:gammaj} implies that if for some $i\in [k]$ we have $a_i\in I_j$, then $m_{j}\leq (1+\delta)^{3}a_i$. Since each $I_j$ contains exactly $n_j$ elements of sequence $\{a_{i}\}_{i=1}^{k}$ (by Definition~\ref{def:njaij}), we infer
		$$n_jm_j=\sum_{i\in[t], \; a_{i}\in I_j}m_j\leq \sum_{i\in[t], \; a_{i}\in I_j}(1+\delta)^{3}a_i.$$
		Since every member of $\{a_i\}_{i=1}^{k}$ is in some $I_j$, we have
		$$S_1=\sum_{j=1}^{t}n_jm_j\leq \sum_{i=1}^{k}(1+\delta)^3a_i\stackrel{(\ref{ineq:many})}{\leq} (1+\delta)^{3}\left(1-\frac{\varepsilon}{2}\right)n\leq \left(1-\frac{3\varepsilon}{8}\right)n.$$
		Now we estimate $S_2$. Recall that sequences $\{m_{j}\}_{j=1}^{t}$ is increasing. We also have $n^{1-\varepsilon}\in I_{t}$, so Claim~\ref{claim:ineq:gammaj} implies that $m_{t}\leq (1+\delta)^3n^{1-\varepsilon}$. So for all $j\in[t]$, we have $m_j\leq (1+\delta)^3n^{1-\varepsilon}$. Therefore	
		$$S_2=\sum_{j=1}^{t}\left(1+ n^{\lambda+1/2}m_j^{-1/2}\right)m_{j}= \sum_{j=1}^{t}m_{j}+\sum_{j=1}^{t} n^{\lambda+1/2}m_j^{1/2}=O(tn^{1-\varepsilon})+O(tn^{\lambda+1/2}n^{1/2-\varepsilon/2}).$$
		Recall that $t=O(\ln n)$, so for large enough $n$ we have 
		$$S_2=o(n)\leq \frac{\varepsilon}{8}n.$$
		Finally, combining the upperbounds on both sums we get	
		$$\sum_{j=1}^{t}b_jm_j\leq S_1+S_2\leq  \left(1-\frac{3\varepsilon}{8}\right)n+\frac{\varepsilon}{8}n=\left(1-\frac{\varepsilon}{4}\right)n.$$
		\end{proof}
	{\bf Step 3.} 	
	Now, let $M_1, \ldots, M_\ell$ be disjoint perfect matchings on $K_{2n}$. Assume that $G$ is a graph  formed by $\ell$ disjoint perfect matchings $M_1, \ldots, M_\ell$ on vertex set $X$ of size $2n$.
	\begin{definition}
		For a collection of matchings $\mathcal{M}\subseteq\{M_1,\ldots, M_b\}$ and a set $X'\subseteq X$ define $G[X',\mathcal{M}]$ to be the subgraph of $G$ formed by the edges $e=xy$, where $x,y\in X'$ and $xy\in M$ for some $M \in \mathcal{M}$.
	\end{definition}
	
	Let $\mathcal{M}_1,\ldots, \mathcal{M}_t$ be a partition of a set $\{M_1,\ldots, M_k\}$ such that $|\mathcal{M}_j|=b_j$ for all $j\in [t]$. Recall that elements of $\mathcal{M}_{j}$ were also labeled by $M_{j}^{i}$, so that $\mathcal{M}_{j}=\{M^1_{j}, \ldots, M^{b_j}_{j}\}$ for all $j\in[t]$.

	The following Claim allows us to construct vertex disjoint subgraphs $G_1,\ldots, G_t$ of $G$, to which we will subsequently apply Lemma~\ref{lemma:AKS2}.
	
	To this end, for $j\in[t]$ let $$D_j=b_j^2m_j/n \;\;\;\; \text{and} \;\;\;\; \Delta_j=K\sqrt{D_j\ln D_j}.$$ 
	\begin{claim}\label{claim:split}
	There is partition of $X$ into parts $X_0, X_1,\ldots,X_t$, such that for graphs $G_1,\ldots, G_t$, defined by $G_j=G[X_j,\mathcal{M}_j]$, the following holds for all $j\in[t]$
		\begin{itemize}
			\item $deg_{G_{j}}(v)=D_j\pm \Delta_j$ for all $v\in V(G_j)$.
			\item $|M_{j}^{i}\cap E(G_j)|=D_jm_j\pm \frac{\Delta_jm_j}{2}$ for all $i\in[b_j]$.
		\end{itemize}
	\end{claim}
	\begin{proof}
		
		Define $p_j=m_jb_j/n$ for $j\in [t]$. Then $D_j=p_jb_j$ for all $j\in [t]$. Note that Claim~\ref{claim:ineq:p_j} implies that $\sum_{j=1}^{t}p_j\leq 1-\varepsilon/4$, so each $p_j\in[0,1)$.
		Set $p_0=1-\sum_{j=1}^{t}p_j$. Consider random partition of $X$  into parts $X_0, X_1,\ldots,X_t$ such that each vertex $x\in X$ ends up in set $X_j$ with probability $p_j$ independently of other vertices.
		
		For all $j\in[t]$ and $v\in V(G_j)$, degree of $v$ in $G[X,\mathcal{M}_j]$ is equal to $|\mathcal{M}_j|=b_j$. So $deg_{G_{j}}(v)\sim Bi(b_j, p_j)$ for all $j\in[t]$ and $v\in V(G_j)$. 
		
		Now, for any edge, probability that both of its end point belong to $V(G_j)$ is equal to $p_j^2$. So $|M\cap E(G_j)|\sim Bi(n, p_j^2)$ for all $j\in [t]$ and $M\in \mathcal{M}_j$. 
		
		Notice that if $b_j=0$ (in this case $n_j=0$), then $p_j=0$,  $|X_j|=0$, so in this case $G_j$ is an empty graph and it satisfies both conclusions of the Claim. From now on we may assume that all graphs $G_j$ are such that $b_j\neq 0$.
				
		Calculating the expectation of binomial random variable gives for all $j\in[t],$ $v\in V(G_j)$ and $i\in[b_j]$:
		\begin{itemize}
			\item $\mathbb{E}(deg_{G_{j}}(v))=b_jp_j=b_j^2m_j/n=D_j$
			\item $\mathbb{E}(M^{i}_j\cap E(G_j))=np_j^2=m_j(b_jp_j)=m_jD_j.$
		\end{itemize}
	
		Recall that if $b_j\neq 0$, then $b_j\geq m_j^{-1/2}n^{\lambda+1/2}$ by Definition~\ref{def:bj}, so \begin{equation}\label{ineq:Dj}
			D_j=\frac{b_j^2m_j}{n}\geq n^{2\lambda}.
		\end{equation}
		So for any $j\in[t]$ and $v\in V(G_j)$ applying Chernoff bound to $deg_{G_{j}}(v)$ gives that for large enough $n$
		$$\mathbb{P}(|deg_{G_{j}}(x)-D_j|> \Delta_j)\leq 2e^{-K^2\ln D_j/4}\stackrel{(\ref{ineq:Dj})}{\leq} 2n^{-\frac{K^2\lambda}{2}}.$$
		Similarly, for all $j\in [t]$ and $i\in [b_j]$
		$$\mathbb{P}(||M_j^{i}\cap E(G_j)|-m_jD_j|>\Delta_jm_j/2)\leq 2e^{-K^2m_j\ln D_j/16}.$$
		Recall that $m_j\geq 1$ for all $j\in [t]$, so
		$$\mathbb{P}(||M_j^{i}\cap E(G_j)|-m_jD_j|>\Delta_jm_j/2)\leq 2e^{-K^2\ln D_j/16}\stackrel{(\ref{ineq:Dj})}{\leq} 2n^{-\frac{K^2\lambda}{8}}.$$
		Finally, let $p$ be the probability that a random partition satisfies the conclusion of the Claim for all graphs $G_j$ with $j\in[t]$. There are at most $2tn$ random variables $deg_{G_j}(v)$ and at most $tn$ random variables $|M_j^{i}\cap E(G_j)|$. So by union bound
		$$p\geq 1-t(2n)2n^{-\frac{K^2\lambda}{2}}-t(n)2n^{-\frac{K^2\lambda}{8}}.$$
		Recall that $K=4/\lambda$ and $t=O(\ln n)$, so for $n$ large enough we have $p= 1-o(1)$.
		
		Therefore the random partion above satisfies the conclusion of Claim~\ref{claim:split} with high probability.
	\end{proof}
	Let $X$ be partitioned as in Claim~\ref{claim:split}, then the corresponding graphs $G_j$ with $j\in [t]$ satisfy the assumptions of the Lemma~\ref{lemma:AKS2}.
		
	So, for each $j\in [t]$ we can apply Lemma~\ref{lemma:AKS2} with $G=G_j$, $D=D_j$, $b=b_j$, matchings $\mathcal{M}_j\cap E(G_j)$ and $m=m_j$. We obtain at least $b_j^{'}=b_j\left(1-O( \frac{\ln^{3/2}D_{j}}{D_{j}^{1/2}})\right)$ matchings, say $M^1_j, \ldots, M_{j}^{b_j^{'}}$ and sets $\{A^{i}_j\}_{i=1}^{b^{'}_{j}}$ such that 
	\begin{itemize}
		\item[(L1)] $A_{j}^{i}\subseteq M_{j}^{i}\bigcap E(G_j)$ for all $i\in[b^{'}_j]$,\vspace{5pt}
		\item[(L2)] $|A^{i}_{j}|\geq \frac{m_j}{1+\delta}$ for $i\in[b^{'}_{j}]$,\vspace{5pt}
		\item[(L3)] $\bigcup_{i=1}^{b_j^\prime}A^{i}_j$ is a matching.
	\end{itemize} 
	
	Next, we show the for each $j\in[t]$ there are more matchings $A_{j}^{i}$ than there are elements of $\{a_i\}_{i=1}^{k}$ in $I_j$.
	\begin{claim}\label{claim:b'jnj}
		For all $j\in[t]$ and large enough $n$, we have $b^\prime_j\geq n_j$.
	\end{claim}
	\begin{proof}Recall that,
	$$b^{'}_{j}=b_j\left(1-O\left(\frac{\ln^{3/2}D_j}{D_j^{1/2}}\right)\right)=
				b_j-O\left(\frac{b_j\ln^{3/2}D_j}{D_j^{1/2}}\right).$$
	So by Definition~\ref{def:bj} and since $D_j=b_j^2m_j/n$ we have
	$$b^{'}_{j}\geq n_j+m_j^{-1/2}n^{\lambda+1/2}-O\left(\frac{n^{1/2}\ln^{3/2}D_j}{m_j^{1/2}}\right).$$
	Now, by Claim~\ref{claim:ineq:p_j}, $b_jm_j\leq n$, and since $m_j\geq 1$, we have  $$D_j=b_j^2m_j/n\leq b^{2}_{j}m_{j}^{2}/n\leq n.$$
	So for large enough $n$, we have 
 	$$b^{\prime}_{j}\geq n_j+m_j^{-1/2}n^{\lambda+1/2}-O\left(m_{j}^{-1/2}n^{1/2}\ln^{3/2} n\right) \geq n_j.$$
	\end{proof}
	
	Recall that $n_j$ is equal to the number of elements of sequence $\{a_{i}\}_{i=1}^{k}$ in set $I_j$ and those elements were also labeled as  $\{a^{1}_{j}, \ldots, a_{j}^{n_j}\}$.
	
	We proved that for all $j\in[t]$ there are $b_j^{'}$ sets $\{A_j^{1}, \ldots A_{j}^{b_{j}^{'}}\}$ that satisfy (L1), (L2), (L3). Consequently, by Claim~\ref{claim:b'jnj}, for any $j\in[t]$ and $i\in[n_j]$ set $A_{j}^{i}$ exists. 
	
	\begin{claim}\label{claim:Ajiaji}
		For every $j\in[t]$ and $i\in[n_j]$ we have $|A_{j}^{i}|\geq a_j^{i}$.
	\end{claim}
	\begin{proof}
	By (L2)  we have $|A_{j}^{i}|\geq m_j/(1+\delta),$ so $$|A_{j}^{i}|\geq \left\lceil \frac{m_j}{1+\delta}\right\rceil.$$ Also recall that $a_j^{i}\in I_j$.
	
	If $j\geq C+1$, since $a_j^{i}\in I_j$, Claim~\ref{claim:ineq:gammaj} yields that $\frac{m_j}{1+\delta}\geq a^{i}_j$. Hence $|A_{j}^{i}|\geq a^{i}_{j}$ in this case.
	
	If $j\leq C$, then $m_j=j=a_{j}^{i}$. In Step 1 we defined $C=\lceil\frac{1}{\delta}\rceil$, so due to $j\leq C$ we have $\frac{j}{1+\delta}>j-1$. Therefore in case $j\leq C$ we have $\lceil\frac{j}{1+\delta}\rceil=j$, and so 
	$$|A_{j}^{i}|\geq \left\lceil \frac{m_j}{1+\delta}\right\rceil=\left\lceil \frac{j}{1+\delta}\right\rceil= j=a_{j}^{i}.$$ 
	\end{proof}

	To summarize our steps toward the proof of Theorem~\ref{thm:main1:1}, for a given sequence $\{a_i\}_{i=1}^{k}$ we split it into $t$ sub-sequences $\{a_{j}^{i}\}_{i=1}^{n^{j}}$ with $j\in[t]$. Then, for given $\ell$ disjoint perfect matchings $M_1,\ldots, M_\ell$ on set $X$ of size $2n$ we found $k$ perfect matchings, labeled $\{M_{j}^{i}\}_{j=1, i=1}^{t, n_j}$, and submatchings $A_{j}^{i}\subseteq M_{j}^{i}$. 
	
	From Claim~\ref{claim:Ajiaji} it follows that for all $j\in[t]$ and $i\in[n_j]$ we have $|A_{j}^{i}|\geq a_{j}^{i}$. 
	
	Notice that for a given $j\in[t]$,  $\bigcup_{i=1}^{n_j} A_{j}^{i}$ is a matching (guaranteed by (L3) and Claim~\ref{claim:b'jnj}) so any two $A_{j}^{i}$ and $A_{j}^{i'}$ are vertex disjoint. Also all vertices of each $A_{j}^{i}$ belong to $V(G_j)$ by (L1). Since $G_{1}, \ldots, G_{t}$ are vertex disjoint we get that $A_{j}^{i}$ and $A_{j'}^{i'}$ are also vertex disjoint for $j\neq j'$. Therefore, all $A_{j}^i$'s are vertex disjoint matchings and so $M=\bigcup_{j=1}^{t} \bigcup_{i=1}^{n_j} A_{j}^{i}$ is a matching. This finishes the proof of Theorem~\ref{thm:main1:1}.

	\section{Proof of Theorem~\ref{thm:main2}}\label{section:main2} Proof of Theorem~\ref{thm:main2} relies on Theorem~\ref{thm:main1} and an application of the following Necklace Theorem by Alon~\cite{Alon}.
		
	\begin{theorem}\label{thm:Alon}
		Consider an open necklace (a string of beads) in which beads are of $\ell$ different colors and the number of beads of each color is divisible by $q$.
		
		Then a group of $q$ thieves can cut the necklace between the beads in at most $\ell(q-1)$ places and distribute the obtained interval pieces between themselves in such a way that each thieve will get the same number of beads of each color.    
	\end{theorem}

	For $\varepsilon>0$ and sufficiently large $n$, let $a_1, \ldots, a_k$ be a sequence of integers such that $\sum a_i < (1-\varepsilon)n$ and $a_i\geq n^{\varepsilon}$ for all $i\in[k]$. 	
	
	We may assume that $\{a_i\}_{i=1}^{k}$ is non-decreasing, and that 
	\begin{equation}\label{ineq:ii1}
	 \sum_{i=1}^{k}a_i \geq \frac{n}{3},
	\end{equation}
	Indeed, in case $ \sum_{i=1}^{k}a_i < \frac{n}{3}$, for any $k$ matchings $M_1, \ldots, M_k$ we can find desired subsets $A_1,\ldots, A_k$ greedily.
	
	The goal of this section is to show that sequence $\{a_{i}\}_{i=1}^{k}$ is $(\varepsilon,n)$-happy. Note that Theorem~\ref{thm:main1} implies that for any $\delta>0$, provided $n$ is large, every positive sequence $\{b_i\}_{i=1}^{\ell}$ that satisfies $\sum b_i\leq (1-\delta)n$ and $b_i\leq n^{1-\delta}$ is $(\delta,n)$-happy. 
	
	\subsection{Proof plan} The proof idea Therorem~\ref{thm:main2} is to modify initial sequence $\{a_i\}_{i=1}^{k}$ by a series of `non-happiness' preserving operations until we obtain a happy sequence, this would imply that initial sequence $\{a_i\}_{i=1}^{k}$ is also happy. 
	More specifically, let $\mathcal{A}_0=\{a_{i,0}\}_{i=1}^{k}$ be $\{a_i\}_{i=1}^k$, we recursively define new non-decreasing sequences $\mathcal{A}_1, \ldots, \mathcal{A}_s$ (here $s<k$) such that each sequence $\mathcal{A}_{j}$ consists of $k-j$ elements for $j\in[0,s]$ and the last sequence $\mathcal{A}_s$ is $(\varepsilon, n)$-happy. 
	
	The recursive definition goes as follows: assume that non-decreasing sequence $\mathcal{A}_{j}=\{a_{i,j}\}_{i=1}^{k-j}$ for some $j\in[0,s-1]$ is constructed with the sum of elements at most $(1-\varepsilon/2)n$. In what follows, $\delta$ is a constant much smaller than $\varepsilon$. 
	
	If $a_{k-j,j}$ is at most $n^{1-\delta}$, then we stop the construction and set $s=j$. The sequence $\mathcal{A}_{s}$ is $(\delta, n)$-happy by Theorem~\ref{thm:main1}. For that, we assume that $n$ is at least $n'$, where $n'$ is the number guaranteed by Theorem~\ref{thm:main1} for $\delta$. Finally, as $\delta<\varepsilon$, sequence $\mathcal{A}_{s}$ is also $(\varepsilon,n)$-happy.
	
	 If $a_{k-j,j}$ is larger than $n^{1-\delta}$, we construct new sequence $\mathcal{A}_{j+1}$ such that the elements of $\mathcal{A}_{j+1}$ are taken to be all but the last element of $\mathcal{A}_{j}$ rescaled by some factor (see (\ref{eq:ii4})). 
	 
	 This rescaling factor is chosen so that the sum of elements of $\mathcal{A}_{j+1}$ is larger then sum of elements of $\mathcal{A}_{j}$ by roughly a factor of $(1+\lambda)$, independent of $j$ (Claim~\ref{claim:ail1}(c)). In turn, $\lambda$ is chosen so that the constructed sequences $\mathcal{A}_{0}, \ldots, \mathcal{A}_s$ all have a sum at most $(1-\varepsilon/2)n$ (Claim~\ref{claim:ail2}(a)). 
	 
	 Construction of sequences $\mathcal{A}_{0}, \ldots, \mathcal{A}_s$ stops when the last sequence $\mathcal{A}_s$ is $(\delta,n)$-happy, which will happen eventually as the sequence consisting of single element $\{b\}$ with $b<n$ is $(\varepsilon,n)$-happy by definition.
	
	 Finally, Claim~\ref{claim:nice->nice} shows that if $\mathcal{A}_j$ is $(\varepsilon, n)$-happy then $\mathcal{A}_{j-1}$ is also $(\varepsilon,n)$-happy. Since the construction ends with a sequence $\mathcal{A}_s$ that is $(\varepsilon,n)$-happy, this implies, by repeated application of  Claim~\ref{claim:nice->nice}, that all previous sequences $\mathcal{A}_{j}$ with $j\in[0,s-1]$, in particular $\mathcal{A}_0$, are also $(\varepsilon,n)$-happy. Claim~\ref{claim:nice->nice} is a corollary of a rather general Lemma~\ref{lemma:nice->nice} that allows us to verify that a sequence is $(\varepsilon, n)$-happy, provided  some other sequence is $(\varepsilon,n)$-happy, and the main ingredient of proof of Lemma~\ref{lemma:nice->nice} is Theorem~\ref{thm:Alon}.

	\subsection{Formal proof} Define the auxiliary variables 
	\begin{equation}\label{eq:constants}
	\delta=\frac{\varepsilon}{10}, \qquad \mu=\frac{\varepsilon}{2(1-\varepsilon)}, \qquad S=\lceil 2(1+\mu)n^{\delta} \ln n\rceil, \qquad \lambda=\frac{\ln (1+\mu)}{S}.
	\end{equation}
	
	To explain the choice of variables, $(1+\lambda)$ will determine how much $\sum_{i=1}^{k-j+1}a_{i,j-1}$ is smaller than $\sum_{i=1}^{k-j}a_{i,j}$ for all $j\in [s] $ (Claim~\ref{claim:ail1}(c)). Constant $\mu$ is chosen to guarantee that
	\begin{equation}\label{ineq:ii2}
	(1+\lambda)^{S}\leq (1+\mu)=\frac{1-\varepsilon/2}{1-\varepsilon}.
	\end{equation}
	Inequality~(\ref{ineq:ii2}) follows easily from inequality $1+x\leq e^x$ and is used in proving Claim~\ref{claim:ail2}(a). Finally, $S$ will be shown to be an upper bound on $s$ in Claim~\ref{claim:ail2}(c)
	
	Provided $n$ is large enough, the following inequalities hold:
	\begin{equation}\label{ineq:ii3}
	n^{-2\delta} \leq \lambda \leq 1, \qquad \qquad n^{\delta}\leq S\leq n^{2\delta}.
	\end{equation}
	
	{\bf Construction of sequences $\mathcal{A}_j$.} Recall that we assumed that $\{a_i\}_{i=1}^{k}$ is non-decreasing. We now construct $s+1$ non-decreasing sequences $\mathcal{A}_{j}=\{a_{i,j}\}_{i=1}^{k-j}$ with $j\in[0,s]$. Set $\mathcal{A}_0=\{a_{i}\}_{i=1}^{k}$ and assume that sequence $\mathcal{A}_{j}=\{a_{i,j}\}_{i=1}^{k-j}$ was constructed for some integer $j$. Set $$\s{j}=\sum_{i=1}^{k-j}a_{i,j}.$$ If $a_{k-j, j}>n^{1-\delta}$, then define 
	$$m_{j}=\frac{a_{k-j,j}}{\s{j}},$$
	and for $i\in [k-j-1]$ define the terms of sequence $\mathcal{A}_{j+1}=\{a_{i,j+1}\}_{i=1}^{k-j-1}$ by
	\begin{equation}\label{eq:ii4}
		a_{i,j+1}=\left \lfloor \frac{(1+\lambda)}{1-m_j}a_{i,j}\right \rfloor.
	\end{equation}
	If $a_{k-j,j}\leq n^{1-\delta}$ or $j=k-1$, the construction stops and we set $s=j$. 
	
	We start with proving a few straightforward, but slightly technical observations about constructed sequences.

	\begin{claim}\label{claim:ail1} For the sequences  $\mathcal{A}_0, \ldots, \mathcal{A}_{s}$ we have the following:
		\begin{itemize}
			\item[(a)] $a_{i,j}\leq a_{i,j+1}$ for all $j \in [0,s-1]$ and $i\in [k-j-1]$.
			\item[(b)] For each $j\in[0,s]$ sequence $\mathcal{A}_j$ is non-decreasing.
			\item[(c)] For each $j \in [0,s-1]$
			$$\left(1+\frac{\lambda}{2}\right)\s{j}\leq \s{j+1} \leq (1+\lambda)\s{j}.$$
		\end{itemize}
	\end{claim}
	\begin{proof}[Proof of Claim] The proof is straightforward and uses equation~(\ref{eq:ii4}) and the fact that $n$ is sufficiently large.
			
		\par {\bf Part (a).}
		The proof is by induction on $j$. The base case $j=0$ and inductive step are proved simultaneously. Assume that for some $h\in[0,s-1]$ we proved the statement for all values of $j< h$, that is assume that for all $j\in[h]$ we have $a_{i,j}\leq a_{i,j+1}$ for all $i\in[k-j+1]$. Then for all $i\in[k-h]$ we have 
		$$a_{i,h}\geq a_{i,h-1}\geq \ldots \geq  a_{i,0}\geq n^{\varepsilon}$$ and so
		$$a_{i,h+1} \stackrel{\text{(\ref{eq:ii4})}}{\geq} \frac{(1+\lambda)}{1-m_h}a_{i,h}-1 \geq (1+\lambda)a_{i,h}-1\geq a_{i,h}+(\lambda n^\varepsilon-1).$$
		Then, by inequality~(\ref{ineq:ii3}) and provided $n$ is sufficiently large
		$$a_{i,h+1}\geq a_{i,h}+(n^{\varepsilon-2\delta}-1)>a_{i,h}.$$
	
		{\bf Part (b).} Again, the proof is by induction on $j$. The base case $j=0$ holds because $\mathcal{A}_0=\{a_{i,0}\}_{i=1}^k$ is assumed to be non-decreasing. For some $j\in[s]$, assume that $a_{i,j-1}\leq a_{i+1,j-1}$ holds for all $i\in [k-j]$. 
		
		Our goal is to show that $a_{i,j}\leq a_{i+1,j}$ for all $i\in [k-j-1]$. If for some $i\in[k-j-1]$ we have that $a_{i,j-1}=a_{i+1,j-1}$, then $a_{i,j}=a_{i+1,j}$ by (\ref{eq:ii4}). 
		
		So we may assume that $a_{i+1,j-1}\geq a_{i,j-1}+1$. Then
		$$ a_{i+1,j} \stackrel{\text{(\ref{eq:ii4})}}{\geq} \frac{(1+\lambda)}{1-m_j}a_{i+1,j-1}-1\geq \frac{(1+\lambda)}{1-m_j}(a_{i,j-1}+1)-1 \geq \frac{(1+\lambda)}{1-m_j}a_{i,j-1} \stackrel{\text{(\ref{eq:ii4})}}{\geq}  a_{i,j}.$$
		Therefore for all  $i\in [k-j-1]$ we have  $a_{i,j}\leq a_{i+1,j}$.
		
		{\bf Part (c).} Note that $\frac{1}{1-m_j}\sum_{i=1}^{k-j-1}a_{i,j}$ is equal to $\s{j}$ by definition of $m_j$. Then,
		$$\s{j+1} \stackrel{\text{(\ref{eq:ii4})}}{\leq} \sum_{i=1}^{k-j-1}\frac{(1+\lambda)}{1-m_j}a_{i,j}=(1+\lambda)\s{j}.$$ 
		Similarly, for the lower bound, 
		$$\s{j+1} \stackrel{\text{(\ref{eq:ii4})}}{\geq} \sum_{i=1}^{k-j-1}\left(\frac{(1+\lambda)}{1-m_j}a_{i,j}-1\right)=\left(1+\frac{\lambda}{2}\right)\s{j}+\frac{\lambda}{2}\s{j}-(k-j-1).$$ 
		Now, Claim~\ref{claim:ail1}(a) implies that $a_{i,j}\geq a_{i,0}\geq n^{\varepsilon}$ for all $j\in[0,s]$, $i\in[k-j]$. Using inequality (\ref{ineq:ii3}), provided $n$ is sufficiently large, we get
		$$\s{j+1} \geq \left(1+\frac{\lambda}{2}\right)\s{j}+\frac{n^{-2\delta}}{2}(k-j)n^{\varepsilon}-(k-j-1)\geq \left(1+\frac{\lambda}{2}\right)\s{j} .$$ 

		\end{proof}
		
		The next Claim, in particular, shows that the construction of sequences $\mathcal{A}_j$'s ends after at most $S$ steps.
		
		\begin{claim}\label{claim:ail2} Let the sequences $\mathcal{A}_0,\ldots, \mathcal{A}_s$ and auxiliary $m_{0}, \ldots, m_{s-1}$ be constructed as described above, then
			\begin{itemize}
				\item[(a)] $\s{j}\leq (1-\frac{\varepsilon}{2})n$ for all $j \in [0, s]$.
				\item[(b)] $m_{j}\geq n^{-\delta}$ for all $j \in [0, s-1]$. 
				\item[(c)] $s< S$.
			\end{itemize}
		\end{claim}
		\begin{proof}
		We start with proving parts (a) and (b) with additional assumption that $j\leq S$. Then we prove part (c), which in turn implies parts (a), (b) for all $j\leq s$.	
		
		{\bf Part (a).} Assume $j \leq S$, then by repeatedly applying Claim~\ref{claim:ail1}(c) and inequality~(\ref{ineq:ii2}) we get
		$$\s{j}\leq (1+\lambda)^{j}\sigma_0= (1+\lambda)^{S}\sum_{i=1}^{k}{a_{i,0}}\leq \frac{1-\varepsilon/2}{1-\varepsilon}(1-\varepsilon)n=\left(1-\frac{\varepsilon}{2}\right)n.$$
		
		{\bf Part (b).} Assume $j \leq S$. If $m_j$ is defined, then $a_{k-j,j}\geq n^{1-\delta}$ so by Claim~\ref{claim:ail2}(a) 
		$$m_j=\frac{a_{k-j,j}}{\s{j}}\geq \frac{n^{1-\delta}}{n} \geq n^{-\delta}.$$
		
		{\bf Part (c).} Assume contrary, that is that the sequence $\mathcal{A}_S=\{a_{i,S}\}_{i=1}^{k-S}$ is defined. For $h\in [0,S-1]$ and $j\in[0, h]$ define 
		$$f_{h,j}=\frac{a_{k-h,j}}{\sum_{i=1}^{k-h}a_{i,j}}.$$
		
		Note that $f_{h,h}=m_h$. In what follows we will show that $f_{h,h}\leq (1+\mu)f_{h,0}$ for all $h\in [0,S-1]$ and will use this to show that $a_k>n$, deriving a contradiction.
		
		We will make use of a lower bound for each $a_{i,j}$ with $j\in [S]$, $i\in[k-j]$.  Claim~\ref{claim:ail1}(a) implies that $a_{i,j}\geq a_{i,0}\geq n^{\varepsilon}$ and so
		$$a_{i,j}\stackrel{\text{(\ref{eq:ii4})}}{\geq} \frac{1+\lambda}{1-m_{j-1}}a_{i,j-1}-1\geq \frac{1+\frac{\lambda}{2}}{1-m_{j-1}}a_{i,j-1}+\frac{\lambda}{2}a_{i,j}-1\stackrel{\text{(\ref{ineq:ii3})}}{\geq} \frac{1+\frac{\lambda}{2}}{1-m_{j-1}}a_{i,j-1}+\frac{n^{-2\delta}}{2}n^\varepsilon-1.$$
		Therefore, for large $n$,
		\begin{equation}\label{eq:aijp21}
		a_{i,j} \geq \frac{1+\frac{\lambda}{2}}{1-m_{j-1}}a_{i,j-1}.
		\end{equation}	
		Then for any $h\in[0,S-1]$, $i\in[0,h]$ by (\ref{eq:aijp21}) and definition of $a_{k-h,j}$~(\ref{eq:ii4}),
		$$f_{h,j}=\frac{a_{k-h,j}}{\sum_{i=1}^{k-h}a_{i,j}}\leq \frac{\frac{1+\lambda}{1-m_{j-1}}a_{k-h,j-1}}{\sum_{i=1}^{k-h}\frac{1+\frac{\lambda}{2}}{1-m_{j-1}}a_{i,j-1}}=\frac{1+\lambda}{1+\frac{\lambda}{2}}f_{h,j-1}.$$
		
		Therefore, $f_{h,j}\leq (1+\lambda)f_{h,j-1}$ for all $h\in[0,S-1]$ and $j\in[h]$. Applying this inequality $h$ times we get,
		$$m_h=f_{h,h}\leq (1+\lambda)^{h}f_{h,0}.$$
		Now, for $h\in[0,S-1]$, by inequality~(\ref{ineq:ii2}),
		$$m_{h}\leq (1+\mu)f_{h,0}.$$
		Therefore, $f_{h,0}\geq \frac{1}{1+\mu}m_h$ and by Claim~\ref{claim:ail2}(b) we get that $f_{h,0}\geq \frac{n^{-\delta}}{1+\mu}$ for all $h\in [0, S-1]$. This means that for all $h \in [0, S-1]$,
		\begin{equation}\label{eq:ak-h}
		a_{k-h}\geq \frac{n^{-\delta}}{1+\mu}\left(a_{1}+\ldots+ a_{k-h}\right).
		\end{equation}
		Thus by successive application of (\ref{eq:ak-h}) for $h\in[0, S-1]$ we obtain a lower bound on $a_k$, 
		\begin{align*}
		a_k& \geq\frac{n^{-\delta}}{1+\mu}(a_1+\ldots+a_k)\\
		& \geq \frac{n^{-\delta}}{1+\mu}\left(1+\frac{n^{-\delta}}{1+\mu}\right)(a_1+\ldots+a_{k-1})\\
		& \geq \frac{n^{-\delta}}{1+\mu}\left(1+\frac{n^{-\delta}}{1+\mu}\right)^2(a_1+\ldots+a_{k-2})\\
		& \vdots \\
		& \geq \frac{n^{-\delta}}{1+\mu}\left(1+\frac{n^{-\delta}}{1+\mu}\right)^S(a_1+\ldots+a_{k-S})	
		\end{align*}
		Now, by using inequality $1+x\geq e^{x/2}$ for $x\in (0,2)$ and the fact that each $a_i$ is at least $n^{\varepsilon}$, for $n$ large enough
		$$a_k\geq \frac{n^{-\delta}}{1+\mu}e^{n^{-\delta}S/(2+2\mu)}n^{\varepsilon}\stackrel{\text{(\ref{eq:constants})}}{\geq}\frac{n^{\varepsilon-\delta}}{1+\mu}e^{\ln n}>n, $$ 
		contradicting the assumption that $s\geq S$.
	\end{proof}
	
	We now show that if $\{a_{i,j}\}_{i=1}^{k-j}$ is happy, than the preceding sequence $\{a_{i,j-1}\}_{i=1}^{k-j+1}$ is also happy. To show that we prove a rather more general lemma that allows to verify that sequence of length $t+1$ is happy, provided a certain sequence of length $t$ is happy.
	
	\begin{lemma}\label{lemma:nice->nice}
		Assume that a sequence $\{b_{i}\}_{i=1}^{t}$ is $(\varepsilon,n)$-happy and that $\frac{p}{q}$ is a positive proper fraction. If $\{a_{i}\}_{i=1}^{t+1}$ is a sequence of integers that satisfies $a_i\leq (1-\frac{p}{q})b_i-q$ for $i\in [t]$ and $a_{t+1}\leq \frac{p}{q}(\sum_{i=1}^{t}b_i)-3(t+1)q$, then $\{a_i\}_{i=1}^{t+1}$ is also $(\varepsilon,n)$-happy.
	\end{lemma} 
	\begin{proof}
		Let $\ell \in[t, (1+\varepsilon)t]$ be the number of disjoint perfect matchings provided by the definition of $\{b_i\}_{i=1}^{t}$ being $(\varepsilon,n)$-happy. We will show that $\{a_{i}\}_{i=1}^{t+1}$ is $(\varepsilon,n)$-happy with $\ell^\prime=\ell+1\in[t+1,(1+\varepsilon)(t+1)]$.
				
		In order to show that sequence $\{a_{i}\}_{i=1}^{t+1}$ is $(\varepsilon,n)$-happy, we need to show that for any $\ell^\prime$ disjoint perfect matchings there are $t+1$ of them $M_1, \ldots M_{t+1}$ and vertex-disjoint matchings $A_i\subseteq M_i$ such that $|A_{i}|\geq a_i$ for all $i\in[t+1]$ (then matching $M$ can be taken to be $M=\bigcup A_i$). 
		
		Let $\ell^\prime$ disjoint perfect matchings of $K_{2n}$ be given and $M_{t+1}$ be one of the matchings. Among $\ell=\ell^\prime-1$ remaining matchings, due to $\{b_{i}\}_{i=1}^{t}$ being $(\varepsilon,n)$-happy, there are $t$ of them $M_1,\ldots, M_t$ and subsets $B_i\subseteq M_i$ for $i\in[t]$ such that all $B_i$'s are vertex disjoint and $|B_i|\geq b_i$ for all $i\in [t]$. 
		
		In what follows we will find vertex disjoint matchings $A_1,\ldots, A_{t+1}$ such that $A_{t+1}\subseteq M_{t+1}$, $A_i\subseteq B_i$ for $i\in [t]$ and
		\begin{equation}\label{eq:AiAt+1}|A_i|\geq \left(1-\frac{p}{q}\right)b_i-q \; \text{for all} \; i\in[t], \; \text{and} \; |A_{t+1}|\geq\frac{p}{q}\left(\sum_{i=1}^{t}b_i\right)-3(t+1)q.\end{equation}
		Existence of such sets $A_1,\ldots, A_{t+1}$ implies that the sequence $\{a_{i}\}_{i=1}^{t+1}$ is $(\varepsilon,n)$-happy. Indeed, if $A_{1}, \ldots ,A_{t+1}$ are such that~(\ref{eq:AiAt+1}) holds, then $|A_i|\geq a_i$ for all $i\in[t+1]$.
		
		In order to satisfy divisibility condition in Theorem~\ref{thm:Alon}, for all $i\in [t]$ let $B'_i$ be a largest subset of $B_i$ such that $q$ divides $|B'_i|$, and $B'_{t+1}$ be a largest subset of $M_{t+1}$ such that $q$ divides $|B'_{t+1}|$. Let $b'_i=|B'_i|$ for $i\in [t+1]$, then 
		\begin{equation}\label{eq:b_iprime}
		b'_{i}\geq b_i-(q-1) \; \text{for} \; i\in [t] \; \text{and} \; b'_{k+1}\geq n-(q-1).
		\end{equation}		
		 We want to find sets $A_i\subseteq B_i^\prime$ such that $|A_i|\geq a_i$ for all $i\in[t+1]$ and $\bigcup_{i=1}^{t+1}A_i$ forms a matching. Since $\bigcup_{i=1}^{t}B_{i}^\prime\subseteq \bigcup_{i=1}^{t}B_i$ is a matching by our assumption, we only need to take care of incidences between edges of $B_{t+1}^\prime$ and $\bigcup_{i=1}^{t}B_{i}^\prime$. In what follows we will call the edges of $\bigcup_{i=1}^{t}B_{i}^\prime$ old and edges of $B_{t+1}^\prime$ new.
		
		The necklace needed to apply Theorem~\ref{thm:Alon} is constructed in the following four steps.
		
		{\bf Step 1.} Let $G$ be a graph that is formed by a union of all $B'_{i}$'s, $i\in[t+1]$. Then $G$ is a union of two matchings $\bigcup_{i=1}^{t}B'_{i}$ and $B'_{t+1}$, and hence each component of $G$ is either a cycle or a path. Colour edge $e$ of $G$ in colour $i$ if $e \in B'_i$. We say that colours in $[t]$ are old and colour $t+1$ is new.
		
		{\bf Step 2.} Let $L(G)$ be a line graph of $G$. The colouring of edges of $G$ induces a colouring of vertices of $L(G)$, vertices of $L(G)$ are properly $[t+1]$ coloured, and all neighbours of a new vertex are coloured with the old colour. Each component of $L(G)$ is also either a path or a cycle.
		
		{\bf Step 3.} Let $G'$ be a subgraph of $L(G)$ obtained from $G$ by deleting an edge from each cycle component. Then $G'$ is a union of paths $P_1, \ldots, P_{h}$. 
		
		{\bf Step 4.} Construct a necklace $N$ by concatenating paths $P_1,\ldots, P_{h}$ (this creates $h-1$ new links/edges in the necklace). The necklace contains $t+1$ types of beads and there are exactly $b'_i$ beads of colour $i$. Note that each bead is identified with a vertex of $L(G)$ and an edge of $G$.
		
		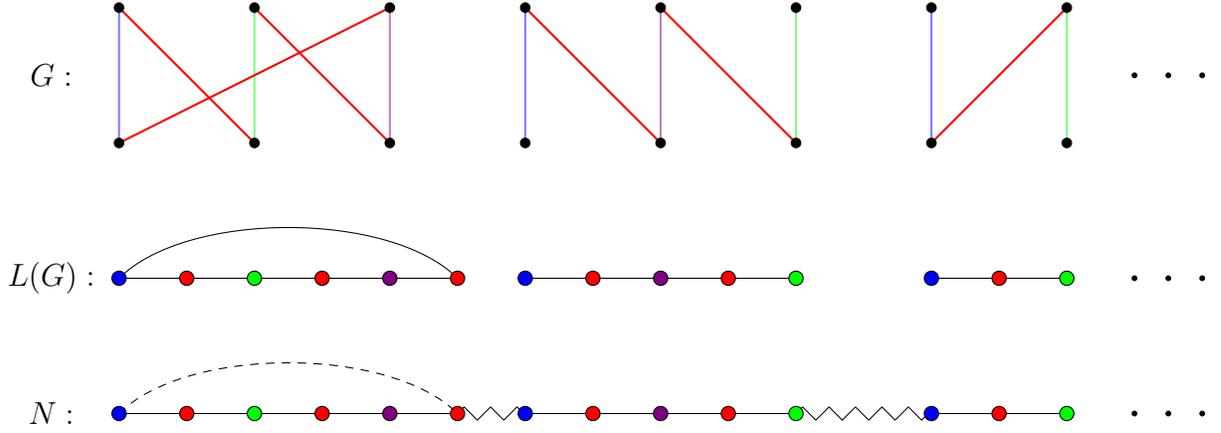
\begin{figure}[H]
			\begin{center}
				\begin{tikzpicture}[line cap=round,line join=round,>=triangle 45,x=1.0cm,y=1.0cm, scale=0.9]

				\clip(-1.7,-4.2) rectangle (16.5,2.1);
				\begin{scope}[xshift=0cm]
				\draw [color=blue, opacity=0.5, thick] (0,0)--(0, 2);
				\draw [color=blue, opacity=0.5, thick] (6,0)--(6, 2);
				\draw [color=blue, opacity=0.5, thick] (12,0)--(12, 2);
				\draw [color=green, opacity=0.5, thick] (2,0)--(2, 2);
				\draw [color=green, opacity=0.5, thick] (10,0)--(10, 2);
				\draw [color=green, opacity=0.5, thick] (14,0)--(14, 2);
				\draw [color=violet, opacity=0.5, thick] (4,0)--(4, 2);
				\draw [color=violet, opacity=0.5, thick] (8,0)--(8, 2);
				\draw [color=red, opacity=1, thick] (0,0)--(4, 2);
				\draw [color=red, opacity=1, thick] (2,0)--(0, 2);
				\draw [color=red, opacity=1, thick] (4,0)--(2, 2);
				\draw [color=red, opacity=1, thick] (8,0)--(6, 2);
				\draw [color=red, opacity=1, thick] (10,0)--(8, 2);
				\draw [color=red, opacity=1, thick] (12,0)--(14, 2);
				\draw [fill=black] (-1,1) circle (0pt) node {$G: $};
				\draw [fill=black] (15,1) circle (1pt);
				\draw [fill=black] (15.5,1) circle (1pt);
				\draw [fill=black] (16,1) circle (1pt);
				\foreach \a in {0,...,7}
				\draw [fill=black] (2*\a,0) circle (2pt);
				\foreach \a in {0,...,7}
				\draw [fill=black] (2*\a,2) circle (2pt);	
				\end{scope}
				
				\begin{scope}[yshift=-2cm]
				\foreach \a in {0,1,2,3,4,6,7,8,9,12,13}
				\draw [color=black] (\a,0)--(\a+1, 0);
				\draw (0,0) .. controls (1,1) and (4,1) .. (5,0);
				\draw [fill=black] (-1,0) circle (0pt) node {$L(G): $};
				\draw [fill=black] (15,0) circle (1pt);
				\draw [fill=black] (15.5,0) circle (1pt);
				\draw [fill=black] (16,0) circle (1pt);
				\foreach \a in {1,3,5,7,9,13}
				\draw [fill=red] (\a,0) circle (3pt);
				\foreach \a in {0,6,12}
				\draw [fill=blue] (\a,0) circle (3pt);	
				\foreach \a in {2,10,14}
				\draw [fill=green] (\a,0) circle (3pt);
				\foreach \a in {4,8}
				\draw [fill=violet] (\a,0) circle (3pt);			
				\end{scope}
				
				\begin{scope}[yshift=-4cm]
				\foreach \a in {0,1,2,3,4,6,7,8,9,12,13}
				\draw [color=black] (\a,0)--(\a+1, 0);
				
				\draw[snake=zigzag] (5,0)--(6, 0);
				\draw[snake=zigzag] (10,0)--(12, 0);
				\draw [dashed] (0,0) .. controls (1,1) and (4,1) .. (5,0);
				\draw [fill=black] (-1,0) circle (0pt) node {$N: $};
				\draw [fill=black] (15,0) circle (1pt);
				\draw [fill=black] (15.5,0) circle (1pt);
				\draw [fill=black] (16,0) circle (1pt);
				\foreach \a in {1,3,5,7,9,13}
				\draw [fill=red] (\a,0) circle (3pt);
				\foreach \a in {0,6,12}
				\draw [fill=blue] (\a,0) circle (3pt);	
				\foreach \a in {2,10,14}
				\draw [fill=green] (\a,0) circle (3pt);
				\foreach \a in {4,8}
				\draw [fill=violet] (\a,0) circle (3pt);			
				\end{scope}
				
				\end{tikzpicture}
			\end{center}
			\caption{Graph $G$ with edges of $B^\prime_{t+1}$ coloured red, line graph $L(G)$ and necklace $N$. Dotted edges were removed from $L(G)$, zigzag links were added to $L(G)$ to form a necklace.}\label{fig:2}
		\end{figure}

		By Theorem~\ref{thm:Alon}, applied with $\ell=t+1$, it is possible to split necklace $N$ between $q$ thieves with at most $(q-1)(t+1)$ cuts so that every thief gets exactly $b'_i/q$ beads of colour $i$. 
		
		With intention to avoid incidences between old and new edges, split the thieves in two 
		groups: one of size $p$ (called new thieves) and the other of size $q-p$ (called old thieves). For $i\in[t]$ let $A_i$ be the collection of edges of colour $i$ that belong to old thieves. Similarly let $A_{t+1}^\prime$ be the collection of edges of colour $t+1$ that belong to new thieves. Since each thieve has has exactly $b_{i}^\prime/q$ edges of colour $i$, then for $i\in [t]$
		\begin{equation}\label{eq:Ai}|A_i|= (q-p)\frac{b'_i}{q}\stackrel{\text{(\ref{eq:b_iprime})}}{\geq} \left(1-\frac{p}{q}\right)(b_i-q)\geq \left(1-\frac{p}{q}\right)b_i-q.
		\end{equation}
		and
		\begin{equation}\label{eq:At+1prime}
		|A^\prime_{t+1}|=p\frac{b_{t+1}^\prime}{q}.
		\end{equation}

		As mentioned above, the sets $A_1,\ldots, A_t$ are vertex disjoint matchings, however it is possible that $A^\prime_{t+1}$ is not vertex disjoint from some or all $A_i$ with $i\in [t]$.  
		
		We say that an edge $e\in A^\prime_{t+1}$ forms a {\emph{conflict pair}} with $f\in \bigcup_{i=1}^{t}A_i$ if $e$ and $f$ share a vertex. Let $A_{t+1}$ be obtained from $A_{t+1}^\prime$ by deleting edges $e$ that are in some conflict pair. Then $A_{t+1}$ is vertex disjoint from $A_1, \ldots, A_t$.
		
		In order to estimate the size of $A_{t+1}$ it remains to estimate the number of conflict pairs.
		\begin{claim}\label{claim:conflictpairs}
			There are at most $2(q-1)(t+1)$ conflict pairs.
		\end{claim}
		\begin{proof}
		Assume that a pair $(e,f)$ is a conflict pair and $e$ is a new edge that belongs to new thieve $T$, while $f$ is an old edge that belong to an old thieve $T'$. We say a conflict pair is of type 1 if $ef$ is an edge in $G'$ (see Figure~\ref{fig:conflictpair}(A)), and of type 2 if $ef$ is not an edge in $G'$ but is an edge in $L(G)$ (see  Figure~\ref{fig:conflictpair}(B)). 
		\begin{figure}[H]
			\centering
			\begin{subfigure}{.4\textwidth}
				\begin{tikzpicture}[line cap=round,line join=round,>=triangle 45,x=1.0cm,y=1.0cm, scale=0.9]
				\clip(-0.5,-0.5) rectangle (5.5,1);
				
				\foreach \a in {0,1,2,3,4}
				\draw [color=black, opacity=0.5, thick] (\a,0)--(\a+1, 0);
				\draw [dashed] (0,0) .. controls (1,1) and (4,1) .. (5,0);
				\foreach \a in {1,3,5}
				\draw [fill=red] (\a,0) circle (3pt);
				\foreach \a in {0}
				\draw [fill=blue] (\a,0) circle (3pt);	
				\foreach \a in {2}
				\draw [fill=green] (\a,0) circle (3pt);
				\foreach \a in {4}
				\draw [fill=violet] (\a,0) circle (3pt);
				\draw [fill=black] (3,0+0.4) circle (0pt) node {$e$};
				\draw [fill=black] (2,0+0.4) circle (0pt) node {$f$};	
				\draw [color=black, opacity=1, ultra thick] (2,-1)--(3, 1);					
				\end{tikzpicture}
				\caption{Type 1, cut is through $ef$}
			\end{subfigure}
			\begin{subfigure}{.4\textwidth}
				\begin{tikzpicture}[line cap=round,line join=round,>=triangle 45,x=1.0cm,y=1.0cm, scale=0.9]
				\clip(-0.5,-0.5) rectangle (5.5,1);
				
				\foreach \a in {0,1,2,3,4}
				\draw [color=black] (\a,0)--(\a+1, 0);
				\draw [dashed] (0,0) .. controls (1,1) and (4,1) .. (5,0);
				\foreach \a in {1,3,5}
				\draw [fill=red] (\a,0) circle (3pt);
				\foreach \a in {0}
				\draw [fill=blue] (\a,0) circle (3pt);	
				\foreach \a in {2}
				\draw [fill=green] (\a,0) circle (3pt);
				\foreach \a in {4}
				\draw [fill=violet] (\a,0) circle (3pt);
				\draw [fill=black] (5,0+0.4) circle (0pt) node {$e$};
				\draw [fill=black] (0,0+0.4) circle (0pt) node {$f$};	
				\draw [color=black, opacity=1, ultra thick] (1,-1)--(2, 1);								
				\end{tikzpicture}
				\caption{Type 2, cut is made on a path $P_i$}
			\end{subfigure}
			\caption{Two types of conflict pairs.}
			\label{fig:conflictpair}
			
		\end{figure}
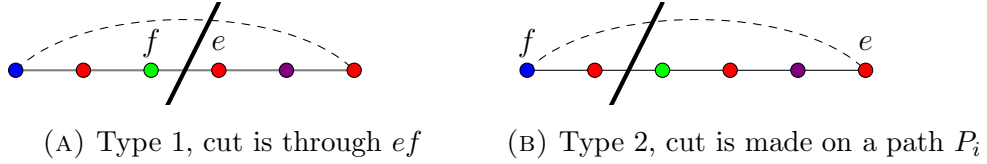
		If $(e,f)$ is type 1 conflict pair, since $e$ and $f$ belong to different thieves, one of the cuts on the necklace was made between $e$ and $f$. So there are at most $(q-1)(t+1)$ type 1 conflict pairs. 
		
		If $(e,f)$ is type 2 conflict pair, then $e$ and $f$ are endpoints of some path $P_i$ in $G^\prime$. Since $e$ and $f$ belong to different thieves, one of the cuts in the necklace was made somewhere on the path $P_i$. There were at most $(q-1)(t+1)$ cuts made, and each cut corresponds to at most one path $P_i$. So there are at most $(q-1)(t+1)$ conflict pairs of type 2.
		
		Therefore there are at most $2(q-1)(t+1)$ conflict pairs in total.\end{proof}
		
		In view of~(\ref{eq:At+1prime}) and Claim~\ref{claim:conflictpairs},
		$$|A_{t+1}|\geq p\frac{b'_{t+1}}{q}-2(t+1)(q-1)\stackrel{\text{(\ref{eq:b_iprime})}}{\geq} \frac{p}{q}(n-q)-2(t+1)(q-1).$$
		Now, since $\{b_i\}_{i=1}^{t}$ is an $(\varepsilon,n)$-happy sequence, then $\sum_{i=1}^{t} b_i\leq n$ and so
		\begin{equation}\label{eq:At+1}
		|A_{t+1}|\geq \frac{p}{q}\left(\sum_{i=1}^{t}b_i\right)-3(t+1)q.\end{equation}
	Finally, (\ref{eq:At+1}) and (\ref{eq:Ai}) imply (\ref{eq:AiAt+1}), establishing that a sequence $\{a_{i}\}_{i=1}^{t+1}$ is $(\varepsilon,n)$-happy.
	\end{proof}
	The following Claim shows that if a sequence $\mathcal{A}_{j}$ is happy, then so is sequence $\mathcal{A}_{j-1}$. 
	\begin{claim}\label{claim:nice->nice}
		If for $j \in[s]$ the constructed sequence $\mathcal{A}_j=\{a_{i,j}\}_{i=1}^{k-j}$ is $(\varepsilon,n)$-happy, then sequence $\mathcal{A}_{j-1}=\{a_{i,j-1}\}_{i=1}^{k-j+1}$ is also $(\varepsilon,n)$-happy.
	\end{claim}
	\begin{proof}
	In view of Lemma~\ref{lemma:nice->nice} it will be sufficient to verify the following inequalities for appropriate $\frac{p}{q}$ in order to show that the sequence $\mathcal{A}_{j-1}$ is $(\varepsilon,n)$-happy:
		$$a_{i,j-1}\leq \left(1-\frac{p}{q}\right)a_{i,j}-q \qquad \text{for all } i\in[k-j], \qquad \text{and} $$
		$$a_{k-j+1, j-1}\leq \frac{p}{q}\s{j}-3(k-j+1)q.$$

		Recall that $m_{j-1}=a_{k-j+1,j-1}/\s{j-1}$. Choose an integer $q\in[\frac{1}{4}n^{4\delta}, \frac{1}{2}n^{4\delta}]$. Let $p=p(j)$ be an integer such that 
		\begin{equation}\label{eq:mpq}
		0\leq m_{j-1}-\frac{p}{q}\leq \frac{1}{q}.
		\end{equation}
		
		Then, by the construction of sequence $\{a_{i,j}\}_{i=1}^{k-j}$, for any $i\in[k-j]$
		
		$$a_{i,j}\geq \frac{1+\lambda}{1-m_{j-1}}a_{i,j-1}-1,$$
		which in turn implies
		\begin{align*}
			a_{i,j-1} \leq \frac{1-m_{j-1}}{1+\lambda}a_{i,j}+\frac{1-m_{j-1}}{1+\lambda} \leq \left(1-\frac{p}{q}\right)\frac{a_{i,j}}{1+\lambda}+1\leq \left(1-\frac{p}{q}\right)a_{i,j}-\frac{\lambda}{1+\lambda}\left(1-\frac{p}{q}\right)a_{i,j}+1.
		\end{align*}
		By Claim~\ref{claim:ail1}(a), $a_{i,j}\geq a_i\geq n^{\varepsilon}$. Therefore by inequality~(\ref{ineq:ii3}), for large enough n,
		\begin{align*}
		a_{i,j-1} \leq 
		\left(1-\frac{p}{q}\right)a_{i,j}-\frac{n^{-2\delta}}{2}\frac{1}{q}n^\varepsilon+1 
		\leq \left(1-\frac{p}{q}\right)a_{i,j}-n^{\varepsilon-6\delta}+1  
		\leq \left(1-\frac{p}{q}\right)a_{i,j}-q,
		\end{align*}
		 proving the first set of inequalities sufficient to use Lemma~\ref{lemma:nice->nice}.
		 
		Now, by inequality~(\ref{eq:mpq}) and Claim~\ref{claim:ail1}(c),  
		\begin{align*}
			a_{k-j+1,j-1}=m_{j-1}\s{j-1}
			 \leq \left(\frac{p}{q}+\frac{1}{q}\right)\frac{\s{j}}{1+\lambda/2} 
		\end{align*}
		which after expanding gives
		\begin{align*}	 
			 a_{k-j+1,j-1}\leq   \frac{p}{q}\s{j}-\frac{p}{q}\frac{\lambda/2}{1+\lambda/2}\s{j}+\frac{\s{j}}{q(1+\lambda/2)}
		\end{align*}

	Now $m_{j-1}\geq n^{-\delta}$ by Claim~\ref{claim:ail2}(b), so for $n$ large enough 
	$$\frac{p}{q}\geq m_{j-1}-\frac{1}{q}\geq n^{-\delta}-4n^{-4\delta}\geq \frac{1}{2}n^{-\delta}.$$
	By inequality~(\ref{ineq:ii3}), $n^{-2\delta}\leq \lambda \leq 1$ and so for large enough $n$,
	\begin{align*}	 
	a_{k-j+1,j-1}\leq   \frac{p}{q}\s{j}-\frac{1}{2}n^{-\delta}\frac{n^{-2\delta}/2}{2}\s{j}+\frac{4\s{j}}{n^{4\delta}}\leq \frac{p}{q}\s{j}-\frac{1}{16}n^{-3\delta}\s{j}.
	\end{align*}
	
	Now, $\s{j}\geq \frac{n}{3}$ by Claim~\ref{claim:ail1}(c) and inequality~(\ref{ineq:ii1}), so
    $$ a_{k-j+1,j-1}\leq \frac{p}{q}\s{j}-\frac{1}{48}n^{1-3\delta}.$$
    Finally, notice that $k\leq n^{1-\varepsilon}$, as $\sum_{i=1}^{k} a_i<n$ and $a_i\geq n^{\varepsilon}$ for $i\in[k]$. Also, recall that $\epsilon=10\delta$, so for large enough $n$,
    
    $$a_{k-j+1, j-1}\leq \frac{p}{q}\s{j}-3n^{1-\varepsilon}n^{4\delta}\leq \frac{p}{q}\s{j}-3(k-j+1)q.$$
	
	Hence, Lemma~\ref{lemma:nice->nice} applied to a sequence $\{a_{i,j-1}\}_{i=1}^{k-j+1}$ implies that sequence  $\{a_{i,j}\}_{i=1}^{k-j}$ is $(\epsilon,n)$-happy.
\end{proof}

By construction, the last sequence $\mathcal{A}_{s}$ is either such that $a_{k-s,s}\leq n^{1-\delta}$ and hence $a_{i,s}\leq n^{1-\delta}$ for all $i\in [k-s]$ (by Claim~\ref{claim:ail1}(b)), or such that $\mathcal{A}_{s}=\{a_{1,k-1}\}$ is one element sequence with $a_{1,k-1}\leq (1-\frac{\varepsilon}{2})n$ (by Claim~\ref{claim:ail1}(a)). In the former case, by Claim~\ref{claim:ail1}(a), $\sigma_{s}<(1-\delta)n$, so the sequence $\mathcal{A}_{s}$ satisfies assumptions of Theorem~\ref{thm:main1} and is $(\delta,n)$-happy. In the latter case $\mathcal{A}_{s}$ is $(\varepsilon,n)$-happy by Definition~\ref{def:enhappy} 
	
In any case, the sequence $\mathcal{A}_s=\{a_{i,s}\}_{i=1}^{k-s}$ is $(\varepsilon,n)$-happy, which in turn, by repeatedly applying Claim~\ref{claim:nice->nice}, implies that $\{a_{i,0}\}_{i=1}^{k}=\{a_i\}_{i=1}^{k}$ is also $(\varepsilon,n)$-happy. This finishes the proof of Theorem~\ref{thm:main2}.	

\section{Trees in Steiner Triple System}\label{section:STS}
We start with noticing that for $\varepsilon,\alpha>0$ and integer $k$ and $n\geq n_0(\varepsilon, \alpha, k)$, if each element of a sequence $a_{1}, \ldots, a_{k}$ is at least $\alpha n$ and $\sum_{i=1}^{n}a_{i}\leq (1-\varepsilon)n$, then we can improve Theorem~\ref{thm:main2}. In particular we show that such sequence $a_{1}, \ldots, a_{k}$ is $(\varepsilon,n)$-happy with $\ell=k$ (recall Definition~\ref{def:enhappy}).
\begin{corollary}\label{cor:main}
	For any $\varepsilon, \alpha>0$ there is $n_0$ such that the following holds for all $n\geq n_0$. Let sequence $a_{1}, \ldots, a_{k}$ be such that $\sum_{i=1}^{k}a_{i}\leq(1-\varepsilon)n$ and $a_{i}\geq \alpha n$ for all $i\in[k]$. Then for  any $k$ disjoint perfect matchings $M_1,\ldots, M_{k}$ of $K_{2n}$ there is a matching $M$ such that $|M \cap M_i|\geq a_{i}$.  	
\end{corollary}
\begin{proof}
	We choose $n_0$ implicitly and let $n\geq n_0$ be sufficiently large.
	
	For $\delta=\min\{1/(\frac{1}{\alpha}+1), \varepsilon\}$ and $n$ sufficiently large we have  $a_{i}\geq n^{\delta}$ for all $i\in [k]$. Then, Theorem~\ref{thm:main2} implies that sequence $\{a_{i}\}_{i=1}^{k}$ is $(\delta, n)$-happy, provided $n$ is large enough. 
	
	Notice that $k\leq \frac{1}{\alpha}$ because $\sum_{i=1}^{k}a_i\leq n$ and $a_i\geq \alpha n$ for all $i\in[k]$.
	Then $\delta\leq 1/(k+1)$ and $a_1,\ldots, a_{k}$ is $(\delta, n)$-happy, so Definition~\ref{def:enhappy} implies that $\ell=k$ and consequently that for any $k$ disjoint perfect matchings $M_1, \ldots, M_{k}$ of $K_{2n}$ there is a matching $M$ such that for all $|M\cap M_{i}|\geq a_i$ for all $i\in[k]$	
\end{proof}

We now define a $(c, a_1,\ldots,a_{k})$-turkey. Recall that a hypertree is a simple $3$-uniform hypergraph in which every two vertices are joined by a unique path. Also notice that a hypertree with $n$ hyperedges has $2n+1$ vertices. A size of hypertree is the number of hyperedges in it. A hyperstar $S$ of size $a$ centered at $v$ is a hypergraph on vertex set $v, v_1, v_2, \ldots, v_{2a}$ with edge set $E(S)=\{\{v,v_{2i-1}, v_{2i}\} : i\in[a]\}$. For integers $c,a_1,\ldots, a_k$ we now define a class of trees which we will call $(c,a_1,\ldots, a_k)$-turkeys.    
\begin{definition}
	Let $C$ be an arbitrary hypertree on $c$ vertices with $k$ specified vertices $v_1,\ldots, v_k$. We say that a 3-uniform simple hypertree $T$ is a $(c, a_1,\ldots,a_{k})$-turkey if $T$ is a union of $C$ with vertex disjoint hyperstars $S_1,\ldots,S_k$ which are centered at $v_1,\ldots, v_k$ and of size $a_1, \ldots, a_k$.

\end{definition}
	\begin{figure}[H]
	\begin{center}
		\begin{tikzpicture}[line cap=round,line join=round,>=triangle 45,x=1.0cm,y=1.0cm, scale=1]

		\clip(-2.8,-1.2) rectangle (5.8,3.8);
		\begin{scope}[xshift=0cm]
		\draw [fill=black] (0,0.5) circle (2pt);
		\draw [fill=black] (1,1) circle (2pt);
		\draw [fill=black] (2,1) circle (2pt);
		\draw [fill=black] (3,0.5) circle (2pt);
		\draw [fill=black] (0.5,-0.5) circle (2pt);
		\draw [fill=black] (1.5,-0.5) circle (2pt);
		\draw [fill=black] (2.5,-0.5) circle (2pt);
		\draw [fill=black] (0,0.5)++(180:1) circle (2pt);
		\draw [fill=black] (0,0.5)++(180:2) circle (2pt);
		\draw [fill=black] (0,0.5)++(150:1) circle (2pt);
		\draw [fill=black] (0,0.5)++(150:2) circle (2pt);
		\draw [fill=black] (1,1)++(100:1) circle (2pt);
		\draw [fill=black] (1,1)++(100:2) circle (2pt);
		\draw [fill=black] (1,1)++(130:1) circle (2pt);
		\draw [fill=black] (1,1)++(130:2) circle (2pt);
		\draw [fill=black] (2,1)++(100:1) circle (2pt);
		\draw [fill=black] (2,1)++(100:2) circle (2pt);
		\draw [fill=black] (2,1)++(70:1) circle (2pt);
		\draw [fill=black] (2,1)++(70:2) circle (2pt);	
		\draw [fill=black] (2,1)++(40:1) circle (2pt);
		\draw [fill=black] (2,1)++(40:2) circle (2pt);	
		\draw [fill=black] (3,0.5)++(15:1) circle (2pt);
		\draw [fill=black] (3,0.5)++(15:2) circle (2pt);
		\draw [fill=black] (3,0.5)++(-15:1) circle (2pt);
		\draw [fill=black] (3,0.5)++(-15:2) circle (2pt);	
		\draw [fill=black] (3,0.5)++(45:1) circle (2pt);
		\draw [fill=black] (3,0.5)++(45:2) circle (2pt);
		\draw [fill=black] (3,0.5)++(-45:1) circle (2pt);
		\draw [fill=black] (3,0.5)++(-45:2) circle (2pt);
		\draw[line width=15pt, color=red, draw opacity=0.25] (0,0.5)--+(180:1)--+(180:2);
		\draw[line width=15pt, color=red, draw opacity=0.25] (0,0.5)--+(150:1)--+(150:2);
		\draw [fill=black] (0,0.5)++(165:2.5) circle (0pt) node {$S_1$};
		\draw[line width=15pt, color=red, draw opacity=0.25] (1,1)--+(100:1)--+(100:2);
		\draw[line width=15pt, color=red, draw opacity=0.25] (1,1)--+(130:1)--+(130:2);
		\draw [fill=black] (1,1)++(115:2.5) circle (0pt) node {$S_2$};
		\draw[line width=15pt, color=red, draw opacity=0.25] (2,1)--+(100:1)--+(100:2);
		\draw[line width=15pt, color=red, draw opacity=0.25] (2,1)--+(70:1)--+(70:2);
		\draw[line width=15pt, color=red, draw opacity=0.25] (2,1)--+(40:1)--+(40:2);
		\draw [fill=black] (2,1)++(70:2.5) circle (0pt) node {$S_3$};
		\draw[line width=15pt, color=red, draw opacity=0.25] (3,0.5)--+(15:1)--+(15:2);
		\draw[line width=15pt, color=red, draw opacity=0.25] (3,0.5)--+(-15:1)--+(-15:2);
		\draw[line width=15pt, color=red, draw opacity=0.25] (3,0.5)--+(45:1)--+(45:2);
		\draw[line width=15pt, color=red, draw opacity=0.25] (3,0.5)--+(-45:1)--+(-45:2);
		\draw [fill=black] (3,0.5)++(0:2.5) circle (0pt) node {$S_4$};
		\draw [line width=15pt, color=blue, draw opacity=0.5] (0.5,-0.5)--(0,0.5)--(0.5,0.3)--(1,1)--(0,0.5)--(1,1)--(0.5,-0.5)--cycle;
		\draw [line width=15pt, color=blue, draw opacity=0.5] (1.5,-0.5)--(1,1)--(1.5,1)--(1.5,0.6)--(1.5,1)--(2,1)--cycle;
		\draw [line width=15pt, color=blue, draw opacity=0.5] (3-0.5,-0.5)--(3-0,0.5)--(3-0.5,0.3)--(3-1,1)--(3-0,0.5)--(3-1,1)--(3-0.5,-0.5)--cycle;
		\draw (1.5,0.2) ellipse (2cm and 1.25cm);
		\draw [fill=black] (0.5,-0.5)++(180:1.2) circle (0pt) node {$C$};
		
		\end{scope}
	
		\end{tikzpicture}
	\end{center}
	\caption{A $(7, 2, 2, 3, 4)$-turkey on 29 vertices, red hyperedges represent hyperstars $S_1, \ldots, S_4$.}\label{fig:2}
\end{figure}
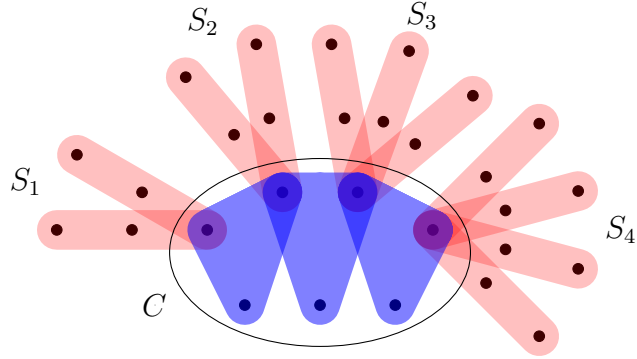

Now we use Corollary~\ref{cor:main} to verify the Conjecture~\ref{conj:ER} for turkeys. 

\begin{theorem}
	For any $\varepsilon, \alpha, c>0$ there is $n_0$ such that the following holds for all $n\geq n_0$. If integers $a_1,\ldots, a_k$ are such $\sum_{i=1}^{k}a_{i}\leq n$ and $a_{i}\geq \alpha n$ for all $i\in[k]$, and $T$ is a $(c, a_1, \ldots, a_{k})$-turkey on $2n+1$ vertices, then any Steiner triple system $S$ on at least $2(1+\varepsilon)n$ vertices contains $T$ as a subhypergraph. 
\end{theorem}
\begin{proof}
	Since $T$ is $(c, a_1, \ldots, a_{k})$-turkey, there is a hypertree $C$ on $c$ vertices, vertices $v_1,\ldots, v_{k}\in V(C)$ and vertex disjoint hyperstars $S_1, \ldots, S_{k}$, such that size of $S_{i}$ is $a_i$ for all $i\in[k]$ and such that $T$ is obtained by identifying centers of stars $S_1, \ldots, S_{k}$ with $v_1, \ldots , v_k$.
	
	Let $S$ be a Steiner triple system of order at least $2(1+\varepsilon)n$. We can find a copy $C^\prime$ of $C$ in $S$ greedily, provided $n\geq c$. Let vertices $v_1, \ldots, v_{k}$ of $C$ correspond to $u_1,\ldots, u_{k}$ in $C^{\prime}$. It remains to embed stars $S_1,\ldots, S_{k}$ into $S$.
	
	Let $X=V(S)\cup \{w\}$, where $w$ is a new ``fake'' vertex. Then $|X|$ is even, since $|V(S)|$ is always odd (in any Steiner triple system  $|V(S)|\equiv 1,3 \; (\text{mod } 6)$ ). For $i\in[k]$ let $M_{i}$ be a matching on $X$ defined by
	$$M_{i}=\{\{u,v\}: \{u_{i},u,v\}\in E(S)\}\cup \{\{u_{i},w\}\}.$$
	We say that for $i\in[k]$ an edge $\{u,v\}\in M_{i}$ is ``fake'' if $\{u,v\}=\{w,u_i\}$ or $|\{u,v\}\cap V(C^{\prime})|\neq 0$. Notice that there are at most $c$ fake edges in each $M_{i}$, since $M_i$ is a matching and every fake edge contains at least one vertex from $C^{\prime}$.
	
	We now have edge disjoint perfect matchings $M_1, \ldots, M_{k}$ on $|X|=2n^\prime$ vertices with $n^\prime\geq (1+\varepsilon)n$. For all $i\in[k]$ let
	$$a_i^{\prime}=\frac{ (1+\varepsilon/2)n^\prime}{(1+\varepsilon)n}a_{i}.$$
	Then for $\varepsilon^\prime=\frac{\varepsilon}{2(1+\varepsilon)}$, since $\sum_{i=1}^{k}a_{i}\leq n$ we have  $$\sum_{i=1}^{k}a_{i}^{\prime}\leq \frac{1+\frac{\varepsilon}{2}}{1+\varepsilon}n^\prime=(1-\varepsilon^\prime) n^{\prime}.$$
	 
	Moreover for $\alpha^{\prime}=\frac{1+\varepsilon/2}{1+\varepsilon}\alpha$ and all $i\in[k]$, we have 
	$$a_{i}^{\prime}\geq \frac{(1+\varepsilon/2)n^\prime}{(1+\varepsilon)n}\alpha n=\alpha^\prime n^\prime. $$
	
	Therefore, provided $n$ is sufficiently large, by Corollary~\ref{cor:main} applied to $\varepsilon^\prime, \alpha^\prime$ and sequence $a_1^\prime, \ldots, a_{k}^{\prime}$, there is a matching $M$ such that $|M\cap M_i|\geq a_{i}^{\prime}$ for all $i\in[k]$.
	
	For $i\in[k]$ let  $A_{i}$ be obtained from $M\cap M_i$ by deleting fake edges. Then for all $i\in[k]$, provided $n$ is large enough,
	$$|A_{i}|\geq a_{i}^\prime -c\geq \left(1+\frac{\varepsilon}{2}\right)a_{i}-c\geq a_i+\frac{\varepsilon}{2}\alpha n-c\geq a_i.$$
	Finally, for each $i\in [k]$  let $$S_i^\prime=\{\{u_i,u,v\}: \{u,v\}\in A_i\}$$
	be the stars centered at $u_1, \ldots, u_k$. For each $i\in[k]$, size of $S^{\prime}_{i}$ is at least $a_i$. By construction, $A_{i}$'s have no fake edges, so $V(S_{i}^{\prime})\cap V(H)=\{v_{i}\}$ for all $i \in [k]$. Since $\bigcup_{i=1}^{k}A_i\subseteq M$ forms a matching, stars $S_1^\prime, \ldots S_k^{\prime}$ are vertex disjoint.
	
	Therefore $C' \bigcup \left(\bigcup_{i=1}^{k}S^{\prime}_{i}\right)$ is a copy of $T$ in $S$. 
\end{proof}
\section{Concluding remarks}
Note that one may modify the proof of Theorem~\ref{thm:main2} allowing to prove a stronger form of Corollary~\ref{cor:main}, namely that for any rational $\alpha_1,\ldots, \alpha_{k}\in(0,1)$ such that $\sum_{i=1}^{k}\alpha_i\leq 1$ there are integers $K$ and $n_0$ such that for any $n\geq n_0$ and any $k$ disjoint perfect matchings $M_1, \ldots, M_{k}$ of $K_{2n}$ there is a matching $M$ such that $|M_{i}\cap M|\geq \alpha_in-K$. 

Finally, some of the authors of this paper believe that for any $\epsilon>0$ there is $n_0$, such that if $n\geq n_0$, then any sequence $a_1,\ldots,a_k$ that satisfies $\sum_{i=1}^{k} a_i<(1-\epsilon)n$ is $(\epsilon,n)$-happy. The only sequences for which we know this to be true are covered by Theorem~\ref{thm:main}.

\end{document}